\documentclass[12pt]{amsart}

\newtheorem*{theoremA}{Theorem A}
\newtheorem*{corollaryB}{Theorem B}

\newcommand{\cal}[1]{\mathcal{#1}}
\usepackage{john}
\usepackage{hyperref}
\usepackage[titletoc]{appendix}

\usepackage{lineno}
\numberwithin{theorem}{section}

\newcommand{\pst}{\oper{pst}}
\newcommand{\pr}{\oper{pr}}

\usepackage[all,cmtip]{xy}
\title{Paraboline variation over $p$-adic families of $(\varphi,\Gamma)$-modules}
\author{John Bergdall}
\date {\today}
\email{bergdall@math.bu.edu}
\urladdr{http://math.bu.edu/people/bergdall}
\address{John Bergdall\\ Department of Mathematics and Statistics \\ Boston University \\ 111 Cummington Mall \\ Boston, MA 02215\\USA}

\begin{document}
\maketitle 

\begin{abstract}
We study the $p$-adic variation of triangulations over $p$-adic families of $(\varphi,\Gamma)$-modules. In particular, we study certain canonical sub-filtrations of the pointwise triangulations and show that they extend to affinoid neighborhoods of crystalline points. This generalizes  results of Kedlaya, Pottharst and Xiao and (independently) Liu in the case where one expects the entire triangulation to extend. We also study the ramification of weight parameters over natural $p$-adic families.
\end{abstract}

\tableofcontents

\section{Introduction}
Let $p$ be a prime number. Results, and questions, in the $p$-adic Langlands program may be naturally phrased, or asked, by studying $p$-adic families of automorphic forms \cite{Emerton-LocalglobalCompat, Breuil-LocallyAnalyticSocle2}. Regarding $p$-adic automorphic forms in their own right, the variation of the associated $(\varphi,\Gamma)$-modules is a central tool \cite{BellaicheChenevier-Book}. One recent result, see \cite{KedlayaPottharstXiao-Finiteness, Liu-Triangulations}, is that if a $p$-adic family of $(\varphi,\Gamma)$-modules is pointwise completely reducible, i.e. a successive extension of rank one objects, then it is also completely reducible in the family, at least generically. In language perhaps known to the reader, and used in this paper, pointwise triangulations in $p$-adic families extend to triangulations over open dense loci. This follows in the tradition of the famous reducibility of local Galois representations over Hida families \cite{MazurWiles-padicAnalyticFamilies}.

However, there are arithmetically interesting points which fail to lie in global triangulation loci. The most concrete example is a point on the Coleman--Mazur eigencurve \cite{ColemanMazur-Eigencurve} corresponding to the non-unit $p$-stabilization of a $p$-ordinary CM form. Such points lie in the image of one of Coleman's $\theta$-maps \cite{Bergdall-CompanionPoints} and the associated $p$-adic $L$-function exhibits curious behavior \cite{Bellaiche-CriticalpadicLfunctions}.

The aim of this article is to give a detailed account of the variation of $(\varphi,\Gamma)$-modules near classical points in $p$-adic families of automorphic forms, especially those which are known not to lie in a global triangulation locus. Note: we will be concerned with $(\varphi,\Gamma_K)$-modules for $K/\Q_p$ a finite extension, but for the purposes of the introduction we will assume $K = \Q_p$. There is no loss of content, only notation.

\subsection{Trianguline $(\varphi,\Gamma)$-modules}
We use $\cal R$ to denote the Robba ring over $\Q_p$. It is the ring consisting of analytic functions $f$ converging on a $p$-adic half-open  annulus $r(f) \leq \abs{T} < 1$, for some radius $r(f)$ depending on $f$. There are continuous commuting actions of a Frobenius operator $\varphi$ and the group $\Gamma \simeq \Z_p^\x$.

A $(\varphi,\Gamma)$-module $D$ is a finite free $\cal R$-module equipped with commuting, $\cal R$-semilinear actions of $\varphi$ and $\Gamma$, such that $\varphi(D)$ generates $D$ as a $\cal R$-module (see \cite{Berger-Representationp-adique}). The rank one $(\varphi,\Gamma)$-modules are parameterized by continuous characters $\delta$ of $\Q_p^\x$. If $\delta$ is such a character we denote the corresponding $(\varphi,\Gamma)$-module by $\cal R(\delta)$. 

A triangulation of a $(\varphi,\Gamma)$-module $D$ is a filtration $P_{\bullet}$ 
\begin{equation*}
0 = P_0 \sci P_1 \sci \dotsb \sci P_{d-1} \sci P_d = D
\end{equation*}
by $(\varphi,\Gamma)$-submodules, such that each successive quotient $P_i/P_{i-1}\simeq \cal R(\delta_i)$ is a rank one $(\varphi,\Gamma)$-module. The ordered tuple $(\delta_1,\dotsc,\delta_d)$ is called the parameter of $P_{\bullet}$. We say $D$ is trianguline if it may be equipped with a triangulation. Examples coming from Galois representations show there may be many ways to triangulate a given $D$.

By the work of Fontaine, Cherbonnier-Colmez,  Kedlaya and Berger, see \cite{Berger-Representationp-adique}, there is a fully faithful embedding $V \mapsto D_{\rig}(V)$ which associates a $(\varphi,\Gamma)$-module to each continuous, finite-dimensional representation $V$ of $G_{\Q_p}$. If $V$ is crystalline then $D_{\rig}(V)$ is trianguline and, generically, the triangulations of $D_{\rig}(V)$ are in bijection with the orderings of eigenvalues for the crystalline Frobenius acting on $D_{\cris}(V)$.

The global context of trianguline $(\varphi,\Gamma)$-modules, and thus this article, is spaces of $p$-adic automorphic forms. It is a folklore conjecture, generalizing a conjecture of Fontaine and Mazur \cite{FontaineMazur}, that an irreducible $\Q_p$-linear, finite-dimensional representation of the global Galois group $G_{\Q}$ which is
\begin{itemize}
\item unramified at all but a finite set of primes and 
\item trianguline at $p$
\end{itemize} 
should (essentially) appear in a space of finite slope $p$-adic automorphic forms. A precise statement was written down by Hansen recently, see \cite[Conjecture 1.2.3]{Hansen-Overconvergent}. The only known result is due to Emerton and settles the question in dimension two, via local-global compatibility in the $p$-adic Langlands program for ${\GL_2}_{/\Q}$ \cite[Theorem 1.2.4(1)]{Emerton-LocalglobalCompat}. The converse, that global representations attached to $p$-adic automorphic forms are trianguline at $p$ is known in many situations by \cite{KedlayaPottharstXiao-Finiteness, Liu-Triangulations}, including eigenvarieties attached to definite unitary groups. This plays a role in recent conjectures of Breuil \cite{Breuil-LocallyAnalyticSocle2} regarding aspects of the $p$-adic local Langlands program for $\GL_n(\Q_p)$ with $n > 2$.

\subsection{Critical triangulations}
The basic notions of $p$-adic Hodge theory (for example, Hodge--Tate--Sen weights and crystalline objects) extend to the category of $(\varphi,\Gamma)$-modules. If $D = \cal R(\delta)$ then the Hodge--Tate--Sen weight is
\begin{equation*}
\wt(\delta) = - \restrict{\mpder /\gamma}{\gamma = 1} \delta(\gamma).
\end{equation*}
If $D$ is triangulated with parameter $(\delta_1,\dotsc,\delta_d)$ then the Hodge--Tate--Sen weights of $D$ are $\set{\wt(\delta_i)}_{i=1,\dotsc,d}$.  The following definition is key for the statement of our theorem. A more general definition will be given in the text (see Section \ref{subsect:critical-noncritical}).

\begin{definition}
Let $D$ be a crystalline $(\varphi,\Gamma)$-module with distinct Hodge--Tate weights $k_1 < \dotsb < k_d$. If $P \ci D$ is a saturated $(\varphi,\Gamma)$-submodule of rank $i$ then $P$ is called non-critical if the lowest $i$ weights $\set{k_1,\dotsc,k_{i}}$ are the Hodge--Tate weights of $P$.
\end{definition}

Note that $D$ is always a non-critical $(\varphi,\Gamma)$-submodule of itself. We extend the notion of non-critical to a triangulation $P_{\bullet}$ by declaring a triangulation $P_{\bullet}$ to be non-critical if $P_i$ is non-critical for each $i$. This agrees with the original definition of Bella\"iche and Chenevier \cite{BellaicheChenevier-Book}.  More generally, we have the following construction. If $D$ is a crystalline $(\varphi,\Gamma)$-module and $P_{\bullet}$ is a triangulation then we define the non-critical indices 
\begin{equation*}
I^{\nc} = \set{i \st P_i \text{ is non-critical}} = \set{0 = i_0 < i_1 < i_2 < \dotsb < i_s = d}
\end{equation*}
and a filtration (called a parabolization, rather than a triangulation, following Chenevier \cite{Chenevier-InfiniteFern})
\begin{equation*}
P^{\nc}_{\bullet} : 0 = P_0^{\nc} \sci P_1^{\nc} \sci \dotsb \sci P_{s}^{\nc} = D
\end{equation*}
by declaring that $P_{j}^{\nc} = P_{i_j}$ for $j = 0,1,\dotsc,s$. In short, $P^{\nc}_{\bullet}$ only knows the non-critical steps in the triangulation $P_{\bullet}$; it is called the maximal non-critical parabolization of $P_{\bullet}$.

\subsection{Refined families and $p$-adic variation}

Now suppose that $X$ is a reduced rigid analytic space over $\Q_p$. In Section \ref{subsec:refined-families} we will define and consider so-called refined families $(\varphi,\Gamma)$-modules over $X$. For now it suffices to know that a refined family $D_X = \set{D_x}_{x}$ is a family of $(\varphi,\Gamma)$-modules over a relative Robba ring $\cal R_X$ with the following properties:
\begin{itemize}
\item there exists continuous characters $\delta_i: \Q_p^\x \goto \Gamma(X,\cal O)^\x$ and
\item  Zariski dense sets of points $X_{\cl}^{\nc} \ci X_{\cl} \ci X(\bar \Q_p)$,
\end{itemize}
such that
\begin{itemize}
\item if $x \in X_{\cl}$ then $D_x$ is crystalline with distinct Hodge--Tate weights and
\item if $x \in X_{\cl}^{\nc}$ then $D_x$ is triangulated by a non-critical triangulation $P_{\bullet,x}$ whose parameter is $(\delta_{1,x},\dotsc,\delta_{d,x})$.
\end{itemize}
Here, $\delta_{i,x}: \Q_p^\x \goto L(x)^\x$ is the character with values in the residue field $L(x)$ obtained by post-composing with the evaluation map at $x$. It follows from \cite[Theorem 6.3.13]{KedlayaPottharstXiao-Finiteness} or \cite[Theorem 5.45]{Liu-Triangulations} that in a neighborhood of a point $x_0 \in X_{\cl}$, each $(\varphi,\Gamma)$-module $D_x$ is triangulated by a triangulation $P_{\bullet,x}$, which is essentially canonical as long as $x_0$ is sufficiently generic (see Proposition \ref{prop:pointwise-existence}). But note that the parameter of $P_{\bullet,x}$ may (and will in critical cases) differ from the natural choice $(\delta_{1,x},\dotsc,\delta_{d,x})$.

Our main result shows that despite the possible non-variation of the pointwise parameters, the maximal non-critical parabolization of $P_{\bullet,x_0}$ varies analytically. Strictly speaking, we only defined $P_{\bullet,x}^{\nc}$ for certain points $x \in X_{\cl}$ but there is a way of extending the definition to every point (which requires reference to the family).

\begin{theoremA}[Theorem \ref{thm:main-theorem}]\label{thm:intro-main}
If $D_X$ is a refined family of $(\varphi,\Gamma)$-modules and $x_0\in X_{\cl}$ is very $\varphi$-regular\footnote{This is a technical condition, and is the same as ``sufficiently generic'' above. See Definition \ref{defi:very-phi-regular}.} then there exists an open affinoid neighborhood $x_0 \in U \ci X$ and a filtration
\begin{equation*}
0 = P_0^{\nc} \sci P_{1}^{\nc} \sci \dotsb \sci P_{s-1}^{\nc} \sci P_{s}^{\nc} =  \restrict{D}{U}
\end{equation*}
where each $P_i^{\nc}$ is a refined family of $(\varphi,\Gamma)$-modules over $U$, such that $P_{i,x}^{\nc} = P_{i,U}^{\nc} \tensor_{\cal O(U)} L(x)$ for all $x \in U$.
\end{theoremA}

The result is optimal in the following sense. If $x \in X_{\cl}$ is non-critical then the Hodge--Tate weights of $P_{i,x}$ are $\set{\wt(\delta_{1,x}),\wt(\delta_{2,x}),\dotsc,\wt(\delta_{i,x})}$. Thus for general $x \in X_{\cl}$ Sen's theory of Hodge--Tate weights in families \cite{Sen-VariationHodgeStructure} implies one can only hope that $P_{i,x}$ extends to an affinoid neighborhood of $x$ provided the Hodge--Tate weights of $P_{i,x}$ are $\set{\wt(\delta_{1,x}),\dotsc,\wt(\delta_{i,x})}$, i.e. $P_{i,x}$ can only vary well in a family if it is a non-critical step in the triangulation at $x$.

The history of our result is relatively short. In the case where $x_0$ is non-critical, so an entire triangulation extends to affinoid neighborhoods, Theorem A has two independent proofs, given essentially at the same time. One proof was given by Liu \cite{Liu-Triangulations} using a generalization of Kisin's interpolation of crystalline periods \cite{Kisin-OverconvergentModularForms, BellaicheChenevier-Book} over general affinoid bases. The other proof was given by Kedlaya, Pottharst and Xiao \cite{KedlayaPottharstXiao-Finiteness} and relied explicitly\footnote{The author has been told that Liu's work established some of these results, implicitly.} on the finiteness of Galois cohomology for families of $(\varphi,\Gamma)$-modules (also proven in \cite{KedlayaPottharstXiao-Finiteness}). Neither work makes any general comment on what to expect at a general classical point, especially in the critically triangulated case.

Our technique is inspired by the latter proof \cite{KedlayaPottharstXiao-Finiteness} and separate work of Liu emphasizing the utility of torsion $(\varphi,\Gamma)$-modules \cite{Liu-CohomologyDuality}. To explain this, let's recall the Kedlaya--Pottharst--Xiao proof of Theorem A in the non-critical case.

Under mild regularity assumptions one can check, using the finiteness of Galois cohomology in families, that $\Hom_{(\varphi,\Gamma)}(\cal R_X(\delta_1),D_X)$ is locally free of rank one near $x_0$. Indeed, it may be checked point-by-point, and essentially just at the point $x_0$ and the points in $X_{\cl}^{\nc}$. Choose an everywhere non-vanishing morphism $\bf e: \cal R_X(\delta_1) \goto D_X$. The non-critical hypothesis on $x_0$ means that the specialized morphism $\bf e_{x_0}: \cal R(\delta_{1,x_0}) \goto D_{x_0}$ has saturated image and, thus, $\coker(\bf e_{x_0})$ is free. An easy argument shows that one can shrink $X$ so that $\coker(\bf e)$ is a $(\varphi,\Gamma)$-module over $X$. The rest of the proof of Theorem A is carried about by induction from this case. 

A generalization of this strategy, to critical points, requires us to define an {\em a priori} candidate for the first step $P_1^{\nc}$ of the sought after parabolization. In the non-critical case $P_1^{\nc}$ is handed to us on a platter as $\cal R_X(\delta_1)$. However, there seems to be no easy way to guess $P_1^{\nc}$ ahead of time\footnote{We thank Eugen Hellman for pointing out the following example which concretely illustrates the issue involved. It is possible to construct two families of rank two of $(\varphi,\Gamma)$-modules $D$ and $D'$, over certain reasonable loci on a Coleman-Mazur eigencurve such that 
\begin{itemize}
\item $D_u \simeq D'_u$, and both are \'etale, for all $u$ except one point $u_0$ and
\item $D_{u_0}$ is \'etale but $D_{u_0}'$ is not.
\end{itemize}
Thus, even knowing that $D$ is an extension of two characters on the complement of a point is not enough to determine an extension over the puncture.}.

The key feature of our proof is to embrace the non-saturatedness of the morphism $\bf e_{x_0}$ in the critical case. Thus we proceed by studying more general families of $(\varphi,\Gamma)$-modules whose fibers have $\cal R$-torsion in them. We call such modules generalized $(\varphi,\Gamma)$-modules following Liu \cite{Liu-CohomologyDuality}. In Section \ref{sec:triangulations}, we give a generalization of the notion of a triangulation of a generalized $(\varphi,\Gamma)$-modules which is well-adapted to attacking Theorem A. The novelty of our proof is the introduction of torsion into the fibers at {\em every point} in order to canonically describe a candidate for $\coker(P_1^{\nc} \goto D)$, thus producing $P_1^{\nc}$ as needed for Theorem A.

\subsection{Ramification of weights}
Let us finish by mentioning an auxiliary result we prove here on the ramification of weights in refined families. Let $D_X$ be a refined family and $x_0 \in X_{\cl}$. Consider the triangulation $P_{\bullet,x_0}$ and denote its parameter by $(\twid \delta_{1,x_0},\dotsc,\twid\delta_{d,x_0})$. As shown in the text, the lists of distinct weights $\set{\wt(\twid \delta_{i,x_0})}$ and $\set{\wt(\delta_{i,x_0})}$ are the same. Thus we may define a permutation $\pi_{x_0}$ by the formula $\wt(\twid \delta_{i,x_0}) = \wt(\delta_{\pi_{x_0}(i),x_0})$. We remark that $\pi_{x_0} = \id$ if and only if $x_0$ is non-critical.

We now let $T_{x_0}X$ be the Zariski tangent space to $X$ at $x_0$. If $f$ is the germ of a function at $x_0$ and $v$ is a tangent vector, we let $\nabla_v(f)$ denote the directional derivative of $f$ with respect to $v$. Our second theorem is that certain differences of weights are constant in every tangent direction.
\begin{corollaryB}[Theorem \ref{theorem:ramification}]
If $D_X$ is a refined family over $X$ and $x_0 \in X_{\cl}$ is very $\varphi$-regular then
\begin{equation*}
\nabla_{v}\left( \wt(\delta_{\pi_{x_0}(i), u}) - \wt(\delta_{i,u}) \right) = 0
\end{equation*}
for all $i=1,\dotsc,d$ and $v \in T_{x_0}X$. In particular, if $x_0$ is critical then the weight map ramifies.
\end{corollaryB}

Theorem B was noticed independently by the author and Breuil. We reproduce an argument similar to \cite[Th\'eor\`eme 9.7]{Breuil-LocallyAnalyticSocle2} in Section \ref{sec:ramification}. Our proof will make use Liu's results on crystalline periods \cite{Liu-Triangulations}. But we note that such a theorem could have been proven using only the (infinitesimal) study of crystalline periods in the weakly refined families  $\wedge^i D$ (as in \cite[Section 4.3]{BellaicheChenevier-Book}, for example) combined with deformation calculations similar to \cite[Proposition 2.4]{Bergdall-CompanionPoints}. 

\subsection{Organization}
Section 2 briefly recalls the theory of $(\varphi,\Gamma)$-modules and the important theorems. Section 3 introduces triangulations and parabolizations, including a definition for torsion $(\varphi,\Gamma)$-modules. Section 4 is a digression into the theorems of \cite{KedlayaPottharstXiao-Finiteness} and applications. Section 6 contains our result on the variation of parabolizations in $p$-adic families, and Section 5 plays a supporting role. Finally we study the ramification of the weight parameters in Section 7. A short appendix is included to deal with a ``relative'' version of Nakayama's lemma.

\subsection{Notations and conventions}\label{subsec:notations}
Throughout the text we will make the following conventions. They will follow \cite{KedlayaPottharstXiao-Finiteness} closely. 

We fix an algebraic closure $\bar \Q_p$ and a $p$-adic valuation on $\bar \Q_p$ so that $\abs{p} = p^{-1}$. 

$K$ will always denote a finite extension $\Q_p$. We let $F$ be the maximal subfield of $F$ unramified over $\Q_p$, $f_K = (F:\Q_p)$ the inertial degree of $K$ and $e_K$ the ramification index of $K$.

We will let $K_\infty = \dirlim_n K(\zeta_{p^n})$ be the extension obtained by adjoining to $K$ all the $p$-power roots of unity. The maximal absolutely unramified subextension of $K_\infty$ is denoted by $F'$. If $H_K = \Gal(\bar \Q_p/K_\infty)$ then we define $\Gamma_K = G_K/H_K$. The cyclotomic character $\Gamma_{\Q_p} \goto \Z_p^\x$ identifies $\Gamma_K$ with an open subgroup of $\Z_p^\x$.

Write $\Sigma_K$ for the set of all the embeddings $K\hookrightarrow \bar \Q_p$. Then $L$ will always denote a finite extension of $\Q_p$ contained in $\bar \Q_p$ such that $\tau(K) \ci L$ for each $\tau \in \Sigma_K$. We allow $L$ to change at will. Note that $L\tensor_{\Q_p} K \simeq \prod_\tau L$ and we denote by $e_\tau$ the idempotent in $L\tensor_{\Q_p} K$ which projects onto the $\tau$-component.

\subsection{Acknowledgements}
The author heartily thanks Rebecca Bellovin, Christophe Breuil, Kiran Kedlaya, Ruochuan Liu, Rob Pollack and Liang Xiao for helpful discussions and questions regarding this work. A special thanks goes to Jay Pottharst for being readily available to discuss his joint work with Kedlaya and Xiao prior to its written debut. We also thank an anonymous referee for many helpful suggestions and corrections. In particular, we owe the crucial flexibility in the definition of finite cohomology in Section \ref{sec:finiteness} to the referee's suggestion. Finally, the bulk of this work was completed as a component of the author's doctoral thesis and it is his pleasure to thank Jo\"el Bella\"iche  for five years of insightful discussions and support. The author was partially supported by NSF award DMS-1402005 during the writing of this article.

\section{Review of $(\varphi,\Gamma_K)$-modules}
We give a short review of (generalized) $(\varphi,\Gamma_K)$-modules, their relationship with Galois representations and the $p$-adic arithmetic theory (cohomology, $p$-adic Hodge theory, etc.). All the notations from Section \ref{subsec:notations} are enforced, including the choice of $L$ for a generic coefficient field containing the image of each embedding of $K$ into $\bar \Q_p$.
\label{sec:review}
\subsection{The Robba ring}\label{subsec:robba-ring}
We quickly remind the reader of the definition of the Robba ring $\cal R$, setting notation for the most part. We note that $K$ is fixed throughout, but that the definition of $\cal R$ depends on $K$ (we simply suppress it from the notation).

For each pair of rational numbers $0 < s \leq r \leq \infty$ we define a $p$-adic annulus
\begin{equation*}
\Af^1_{/F'}[s,r] = \set{T \st p^{-r/(p-1)} \leq \abs{T} \leq p^{-s/(p-1)}}
\end{equation*}
 over $F'$. When $r = \infty$ this is a $p$-adic disc. We also let
 \begin{equation*}
\Af^1_{/F'}(0,r] = \set{T \st p^{-r/(p-1)} \leq \abs{T} < 1}.
\end{equation*}
 be the half-open annulus. We denote by $\cal R^{[s,r]}$ the formal substitution of a certain indeterminate $\pi_K$, arising from the field of norms, for the variable $T$ in the ring of functions on $\Af^1_{/F'}[s,r]$. This is the ring denoted by $\cal R^{[s,r]}(\pi_K)$ in \cite{KedlayaPottharstXiao-Finiteness}.
 
 If $A$ is a $\Q_p$-affinoid algebra then we denote $\cal R_A^{[s,r]} = \cal R^{[s,r]} \hat\tensor_{\Q_p} A$. Let $X = \Sp(A)$ be the associated affinoid space to $A$. Then $\cal R_A^{[s,r]}$ is abstractly isomorphic to the ring of rigid analytic functions on $X^{[s,r]} := \Af^1[s,r]_{/F'} \x_{\Sp \Q_p} \Sp A$. Thus it is a noetherian Banach algebra when equipped with the usual Gauss norm.

If $0 < s < s' \leq r \leq \infty$ then there is an injective restriction morphism $\cal R_A^{[s,r]} \goto \cal R_A^{[s',r]}$ which is flat and has dense image. We then define $\cal R^r_A := \bigintersect_{0 < s \leq r} \cal R_A^{[s,r]}$ and the {\em relative Robba ring over $A$} is
\begin{equation*}
\cal R_A = \bigunion_{0 < r} \cal R_A^r.
\end{equation*}
The ring $\cal R_A^r$ is the global sections on the rigid space $X^r := \Af^1_{/F'}(0,r] \x_{\Sp \Q_p} X$, the relative half open annulus. We will also use the notations $\cal R_X^{r}$ and $\cal R_X$ with the obvious meaning.

Returning to the closed annuli, for any $0 < s \leq r$ there is a continuous action of the group $\Gamma_K\surject \Gal(F'/F)$ on the $F'$ coefficients in $\cal R^{[s,r]}$; we can extend this canonically, up to the choice of $\pi_K$, to the ring $\cal R^{[s,r]}$. When $r$ is sufficiently small there is also an operator $\varphi: \cal R^{[s,r]} \goto \cal R^{[s/p,r/p]}$ called Frobenius which acts on the coefficients in $F'$ via the usual Frobenius action and acts on $\pi_K$ by a choice\footnote{When $K = \Q_p$ then the choice of $\pi_K$ can be made so that $\varphi(\pi_K) = (1+\pi_K)^p - 1$ and $\gamma(\pi_K) = (1 + \pi_K)^{\chi_{\cycl}(\gamma)} - 1$. Moreover, the operator $\varphi$ is defined as soon as $r < {1}$.}, again, canonically up to $\pi_K$. The operator $\varphi$ turns $\cal R^{[s/p,r/p]}$ into a finite free $\cal R^{[s,r]}$-module of rank $p$. Furthermore, $\varphi$ extends to an operator $\varphi: \cal R^{r} \goto \cal R^{r/p}$ and thus also extends to an operator on $\cal R$. When $X$ is an affinoid space we extend the actions of $\Gamma_K$ and $\varphi$ to the relative Robba rings by acting trivially on the coefficients $A$.

There are two ways to view $\cal R_X^{r/p}$ as a module over $\cal R_X^r$, either by the restriction map or the operator $\varphi$. If $Q$ is a module over $\cal R_X^r$ then we denote by $\varphi^{\ast}Q$ the extension of scalars $\varphi^{\ast}Q := Q \tensor_{\cal R_X^{r},\varphi} \cal R_X^{r/p}$ and $\restrict{Q}{(0,r/p]}$ the $\cal R_X^{r/p}$-module obtained by using the restriction map.

\begin{definition}
A generalized $\varphi$-module over $X^{r}$ is a finitely presented $\cal R_X^{r}$-module $Q$ together with an isomorphism $\varphi^{\ast}Q \simeq \restrict{Q}{(0,r/p]}$ of $\cal R_X^{r/p}$-modules. We say that $Q$ is  a $\varphi$-module if $Q$ is also projective.
\end{definition}

We now fix $r_0 > 0$. Since $X$ is affinoid, so is $X^{[s,r]}$. Thus, by Kiehl's theorem \cite[Theorem 9.4.3/3]{BGR}, global sections give an equivalence of categories
\begin{equation*}
\set{\text{finite $\cal R_X^{[s,r]}$-modules}} \longleftrightarrow \set{\text{coherent sheaves on $X^{[s,r]}$}}.
\end{equation*}
Since $X^{r_0}$ is admissibly covered by affinoid opens $\set{X^{[s,r]}}_{0 < s \leq r \leq r_0}$, a coherent sheaf $\cal Q$ on $X^{r_0}$ is the same as a system $\cal Q = (Q^{[s,r]})_{0 < s \leq r \leq r_0}$ of finite $\cal R_X^{[s,r]}$ modules satisfying the obvious compatibilities. The global sections of a sheaf $\cal Q$ may be calculated by
\begin{equation*}
Q = \Gamma(X^{r_0},\cal Q) = \invlim_{0 < s} Q^{[s,r_0]}.
\end{equation*}
If $Q$ is a finitely presented module over $X^{r_0}$ then there is a coherent sheaf defined by the compatible family $Q^{[s,r]} := Q \tensor_{\cal R_X^{r_0}} \cal R_X^{[s,r]}$. We pause to include an auxiliary result on finitely presented modules over $X^{r}$. It applies, in particular, to all generalized $\varphi$-modules.

\begin{lemma}\label{lemma:flat}
Suppose that $Q$ is a finitely presented $\cal R_X^r$-module and $f \in \cal R^r$. Then $Q/f$ is finite projective over $\cal R_X^r/f$ if and only if $Q^{[s,r]}/f$ is finite projective over $\cal R^{[s,r]}/f$ for each each $0 < s \leq r$.
\end{lemma}
\begin{proof}
Since $Q$ is assumed to be finitely presented over $\cal R_X^r$, the same is true for $Q/f$ over $\cal R_X^r/f$ and thus by \cite[Corollary 7.12]{Matsumura-CommutativeRingTheory} it suffices to replace ``projective'' with ``flat'' in the statement of the lemma.

Even without that, one direction is clear: if $Q/f$ is finite projective over $\cal R_X^r/f$ then $Q^{[s,r]}/f = Q/f \tensor_{\cal R_X^r/f} \cal R^{[s,r]}/f$ is finite projective over $\cal R^{[s,r]}/f$ for each $0 < s \leq r$. 

We will now prove the reverse direction, so assume that $Q^{[s,r]}/f$ is flat over $\cal R_X^{[s,r]}/f$ for each $0 < s \leq r$. By \cite[Theorem 7.7]{Matsumura-CommutativeRingTheory} it suffices to show that 
\begin{equation}\label{eqn:need-to-inject}
I\tensor_{\cal R_X^r/f} Q/f \goto Q/f
\end{equation}
is injective for every finitely generated ideal $I \ci \cal R_X^r/f$. Now we use the language of ``co-admissible'' modules originally due to Schneider and Teitelbaum \cite{SchneiderTeitelbaum-padicDistributions}, and we will reference \cite[Section 2.1]{KedlayaPottharstXiao-Finiteness}. By \cite[Lemma 2.1.4(7)]{ KedlayaPottharstXiao-Finiteness} the $\cal R_X^r$-module $\cal R_X^r/f$ is co-admissible, and thus so is $I \ci \cal R_X^r/f$ by \cite[Lemma 2.1.4(6)]{KedlayaPottharstXiao-Finiteness}. Since $Q$ is co-admissible so is $Q/f = \coker(Q\overto{f} Q)$ by \cite[Lemma 2.1.4(5)]{KedlayaPottharstXiao-Finiteness}. Thus \eqref{eqn:need-to-inject} is a morphism of co-admissible $\cal R_X^r$-modules  and hence is injective if and only if
\begin{equation}\label{eqn:need-to-inject-s}
\left(I\tensor_{\cal R_X^r/f} Q/f\right)^{[s,r]} \goto Q^{[s,r]}/f
\end{equation}
is injective for each $0 < s\leq r$. But we see
\begin{equation*}
\left(I\tensor_{\cal R_X^r/f} Q/f\right)^{[s,r]} \simeq \left(I\tensor_{\cal R_{X}^r/f} \cal R_X^{[s,r]}/f\right) \tensor_{\cal R_X^{[s,r]}/f} Q^{[s,r]}/f
\end{equation*}
and then since $\cal R_X^r/f \goto \cal R_X^{[s,r]}/f$ is flat we see that $I\tensor_{\cal R_{X}^r/f} \cal R_X^{[s,r]}/f = I\cdot \cal R_X^{[s,r]}/f$ is an ideal in $\cal R_X^{[s,r]}/f$. In particular, this shows that the map \eqref{eqn:need-to-inject-s} may be identified  with the natural map
\begin{equation*}
I\cdot \cal R_X^{[s,r]}/f \tensor_{\cal R_X^{[s,r]}/f} Q^{[s,r]}/f \goto Q^{[s,r]}/f,
\end{equation*}
which is injective because $Q^{[s,r]}/f$ is flat over $\cal R_X^{[s,r]}/f$ by assumption. This completes the proof.
\end{proof}

Returning to $\varphi$-modules, if $Q$ is a generalized $\varphi$-module then the isomorphism $\varphi^{\ast}Q \simeq \restrict{Q}{(0,r_0/p]}$ translates into the choice of a compatible system of isomorphisms $\varphi^{\ast}Q^{[s,r]} \simeq Q^{[s/p,r/p]}$. In fact, all generalized $\varphi$-modules arise this way.

\begin{proposition}\label{prop:equiv-coh-sheaves}
There is are equivalences of categories
\begin{equation*}
\set{\text{\parbox{4cm}{\centering generalized $\varphi$-modules over $X^{r_0}$}}} \longleftrightarrow \set{\text{\parbox{8.5cm}{\centering coherent sheaves $\cal Q = (Q^{[s,r]})$ on $X^{r_0}$ equipped with naturally compatible isomorphisms $\varphi^*Q^{[s,r]} \iso Q^{[s/p,r/p]}$ }}}.
\end{equation*}
Moreover, the $\varphi$-modules on the left-hand side correspond to sheaves on the right-hand side for which each $Q^{[s,r]}$ is projective.
\end{proposition}
\begin{proof}
We just explained how to go from the left-hand side to the right-hand side. Suppose we start on the right-hand side with a coherent sheaf $\cal Q = (Q^{[s,r]})$. It is clear that all we need to do is show that its global sections are finitely presented. Choose a finite presentation for $Q^{[r_0/p,r_0]}$. By assumption this uniformly (in terms of generators and relations) gives a finite presentation for $(\varphi^{\ast})^nQ^{[r_0/p,r_0]} \simeq Q^{[r_0/p^{n+1}, r_0/p^n]}$. Thus $\cal Q$ is uniformly finitely presented in the sense of \cite[Section 2.1]{KedlayaPottharstXiao-Finiteness}. By \cite[Proposition 2.1.13]{KedlayaPottharstXiao-Finiteness}, the global sections of $\cal Q$ are finitely presented.
\end{proof}

\begin{remark}
The global sections of a coherent sheaf on $X^{r_0}$ need not have any finiteness properties (see the example in \cite[Section 2.1.2, page 380]{Bellovin-padic}). The key for the previous proposition is that the presence of the $\varphi$-operator imposes a uniformly finitely presented condition over the closed annuli.
\end{remark}

\begin{corollary}\label{cor:cat-abelian}
The category of generalized $\varphi$-modules over $X^{r_0}$ is abelian.
\end{corollary}
\begin{proof}
The category of coherent sheaves on $X^{r_0}$ is abelian by \cite[Proposition 9.4.3/2]{BGR}. Thus we just need to show that if $f:\cal Q \goto \cal P$ is a morphism of the corresponding sheaves, then $\ker(f)$ and $\coker(f)$ satisfy the obvious compatibilities. But $\varphi^{\ast}$ is exact because $\varphi$ presents $\cal R_X^{[s/p,r/p]}$ as free over $\cal R_X^{[s,r]}$, so this is immediate.
\end{proof}

There is another way one might think about the $\varphi$-pullback condition on generalized $\varphi$-modules. If $Q$ is a generalized $\varphi$-module over $X^{r_0}$ then the choice of isomorphism $\varphi^{\ast}Q \simeq \restrict{Q}{(0,r_0/p]}$ defines an operator, by $Q \inject \varphi^{\ast}Q \simeq \restrict{Q}{(0,r_0/p]}$ which we also denote by $\varphi$. If $f \in \cal R_X^{r_0}$ and $x \in Q$ then $\varphi(fx) = \varphi(f) \varphi(x)$ and so this version of $\varphi$ is naturally a semi-linear operator (albeit with a different source than target).

\begin{definition}
A generalized $(\varphi,\Gamma_K)$-module $Q$ over $X^{r_0}$ is a generalized $\varphi$-module over $X^{r_0}$ equipped with a continuous $\cal R_X^{r_0}$-semilinear action of $\Gamma_K$ which commutes with $\varphi$. If we drop the word generalized, we insist that $Q$ be a $\varphi$-module.
\end{definition}

Note that we insist that $\Gamma_K$ preserve the radius $r_0$ of the generalized $(\varphi,\Gamma_K)$-module. We are now ready to remove the finite radius assumption.

\begin{definition}
A generalized $(\varphi,\Gamma_K)$-module $Q$ over $X$ is the base change of a generalized $(\varphi,\Gamma_K)$-module $Q = Q^{r_0}\tensor_{\cal R_X^{r_0}} \cal R_X$ over $X^{r_0}$ for some $r_0 > 0$. If we drop the word generalized, we insist that $Q$ be a $(\varphi,\Gamma_K)$-module.
\end{definition}

If $Q$ is a generalized $(\varphi,\Gamma)$-module then there exists an $r_0 > 0$ such that $Q$ arises via base change from a generalized $(\varphi,\Gamma_K)$-module $Q^{r_0}$ over $X^{r_0}$. Thus for any $r_0 > r_0'$, $Q$ also arises geometrically from $Q^{r_0'} := \restrict{Q}{(0,r_0']}$. In particular, if $Q$ is a generalized $(\varphi,\Gamma_K)$-module then we may always take a radius sufficiently small to make sense of the notation $Q^{r_0}$.

By a morphism $f: Q \goto Q'$ of generalized $(\varphi,\Gamma_K)$-modules we mean a continuous $(\varphi,\Gamma_K)$-equivariant morphism of $\cal R_X$-modules. By definition there must exist an $r_0 > 0$ so that $f$ arises from base change of a map $f^{r_0} : Q^{r_0} \goto (Q')^{r_0}$ for $r_0$ sufficiently small. The space of all morphisms will be denote by $\Hom(Q,Q')$. This is also a generalized $(\varphi,\Gamma_K)$-module in the natural way. Taking $Q' = \cal R_X$ we obtain the dual module $Q^\dual$, which we will only use if $Q$ is a $(\varphi,\Gamma_K)$-module.

At various points in Sections \ref{sec:triangulated-families} and \ref{sec:variation} we will need to shrink an affinoid space $X$ to an affinoid subdomain $U  = \Sp B \ci X$. Given such a $U$ and a generalized $(\varphi,\Gamma_K)$-module $Q$ over $X$ we denote by $\restrict{Q}{U}$ the $\cal R_U$-module defined by
\begin{equation*}
\restrict{Q}{U} := Q \hat\tensor_{\cal R_X} \cal R_U = Q\hat\tensor_A B.
\end{equation*} 
Note since $Q$ is finitely presented over $\cal R_A$, the first part of the definition could equivalently be taken to be $\restrict{Q}{U} = Q\tensor_{\cal R_X} \cal R_U$.  We record the following result for later use.
\begin{proposition}\label{prop:exactness-restrict}
If $U \ci X$ is an affinoid subdomain then the association $Q\mapsto \restrict{Q}{U}$ defines an exact functor from the category of generalized $(\varphi,\Gamma_K)$-modules over $X$ to the category of generalized $(\varphi,\Gamma_K)$-modules over $U$.
\end{proposition}
\begin{proof}
The proposition reduces to the same result for generalized $(\varphi,\Gamma_K)$-modules over $X^{r_0}$ for each $r_0 > 0$ (and the corresponding open affinoid subdomain $U^{r_0} \ci X^{r_0}$). Once that reduction has been made, we deduce our result from Proposition \ref{prop:equiv-coh-sheaves}. The main point is that the exactness  follows from the corresponding result for coherent sheaves on rigid spaces (see \cite[Proposition 9.4.1/1]{BGR} for example).
\end{proof}

\subsection{Galois representations}
It will be useful to remind ourselves of the following connection between Galois representations and $(\varphi,\Gamma_K)$-modules. If $A$ is an affinoid $\Q_p$-algebra then by an $A$-linear representation of $G_K$ we mean a finite projective $A$-module $V$ together with a continuous $A$-linear action of the Galois group $G_K$.

\begin{theorem}\label{thm:galois-reps}
Let $X = \Sp(A)$. There is a fully faithful, exact embedding
\begin{equation*}
D_{\rig}: \set{\text{$A$-linear representations $V$ of $G_{K}$}} \hookrightarrow \set{\text{$(\varphi,\Gamma_K)$-modules over $\cal R_X$}}
\end{equation*}
such that
\begin{enumerate}
\item $D_{\rig}$ commutes with base change $A \goto A'$ and
\item when $A$ is finite over $\Q_p$, $D_{\rig}$ is essentially surjective onto the category of \'etale $(\varphi,\Gamma_K)$-modules.
\end{enumerate}
\end{theorem}
The theorem as we've stated it can be read off from \cite[Theorem 3.11, Theorem 0.2]{KedlayaLiu-FamiliesofPhiGammaModules}. The earliest results were proven when $A$ is a field. For that, Fontaine and separately, Cherbonnier and Colmez, gave proofs with the caveat that the Robba ring is replaced with a different ring of analytic functions on affinoid subdomains of discs (see \cite{Fontaine-GFestschriftPaper} and \cite{CherbonnierColmez}). The key step in extending the theorem to the Robba ring as we've discussed it was Kedlaya's theorem on slope filtrations \cite{Kedlaya-p-adicMonodromy}. The family results are more recent and one does not in general have a description of the essential image.

\subsection{Rank one $(\varphi,\Gamma_K)$-modules}\label{subsec:rank-one}
Rank one $(\varphi,\Gamma_K)$-modules over $A$ are parametrized, essentially, by continuous characters $\delta: K^\x \goto A^\x$. Let us recall the construction of $(\varphi,\Gamma_K)$-modules of character type. Note that by \cite[Theorem 6.2.14]{KedlayaPottharstXiao-Finiteness} every rank one $(\varphi,\Gamma_K)$ arises, locally on $A$, from one of character type.

Choose a uniformizer $\varpi_K$ of $K$. We can write $\delta = \delta^{\nr}\delta^{\wt}$ where $\restrict{\delta^{\nr}}{\cal O_K^{\x}} = 1$ and $\delta^{\wt}(\varpi_K) = 1$. Then, $\delta^{\wt}$ extends in a unique manner to the abelianization of the Galois group $G_K^{\ab}$, using the local Artin map, and we denote $\hat \delta^{\wt}$ the corresponding Galois character $\hat \delta^{\wt}: G_K \goto A^\x$. On the other hand, by \cite[Lemma 6.2.3]{KedlayaPottharstXiao-Finiteness} there is a unique rank one free $(F\tensor_{\Q_p} A)$-module $D_{\delta^{\nr}(\varpi_K)}$ equipped with an operator $\varphi$, semilinear with respect to $\varphi \tensor 1$, such that $\varphi^{f_K} = 1\tensor \delta^{\nr}(\varpi_K)$. We give it the trivial $\Gamma_K$-action and define the rank one $(\varphi,\Gamma_K)$-module
\begin{equation*}
\cal R_A(\delta) := (D_{\delta^{\nr}(\varpi_K)} \tensor_{F\tensor_{\Q_p} A} \cal R_A) \tensor_{\cal R_A} D_{\rig}(\hat \delta^{\wt}).
\end{equation*}
This is independent of any choices made and satisfies $\cal R_A(\delta\delta') \simeq \cal R_A(\delta)\tensor_{\cal R_A} \cal R_A(\delta')$. Thus it makes sense to define $D(\delta) := D\tensor_{\cal R_A} \cal R_A(\delta)$ for any generalized $(\varphi,\Gamma_K)$-module over $A$.

Assume now that $A$ is an $L$-algebra for $L$ as in Section \ref{subsec:notations}. If $\delta: K^\x \goto A^\x$ is a character then we can define its weights as follows. The group $K^\x$, as a group over $\Q_p$, has a Lie algebra of dimension $(K:\Q_p) = \sizeof \Sigma_K$. The differential action gives rise to a weights $(\wt_\tau(\delta))_{\tau \in \Sigma_K} \in K\tensor_{\Q_p} A \simeq \prod_{\tau} A_\tau$ such that
\begin{equation*}
0 = \lim_{\substack{a \goto 0\\a \in \cal O_K}} {\abs{ \delta(1+a) - 1 + \sum_{\tau \in \Sigma_K} \wt_{\tau}(\delta)\tau(a) }  \over \abs{a}_K }.
\end{equation*}
It is easy to see that $(\wt_\tau(\delta))_{\tau}$ only depends only $\delta^{\wt}$ in the decomposition of the previous paragraph (thus the notation). We've normalized the weights so that if $z: K^\times \goto K^\times$ is the identity character then $\wt_\tau(z) = -1$ for each $\tau \in \Sigma_K$.

\subsection{$p$-adic Hodge theory}

The definition of weight given above is a special case of extending the usual Fontaine functors $D_{\sen}, D_{\dR}, D_{\cris}$, etc. from $p$-adic Hodge theory to the category of $(\varphi,\Gamma_K)$ modules. In particular, we have the notions of Hodge--Tate--Sen weights, crystalline $(\varphi,\Gamma_K)$-modules, etc.  We won't recall the definitions and will refer to \cite{Berger-Representationp-adique, Berger-EquationsDifferentielles} as needed.

One of Berger's main results \cite[Th\'eor\`eme A]{Berger-EquationsDifferentielles} is that the functor $D_{\pst}(-)$ induces an equivalence
\begin{equation*}
\set{\text{potentially semistable $(\varphi,\Gamma_K)$-modules over $L$}} \overset{D_{\pst}}{\longrightarrow} \set{\text{filtered $(\varphi,N,G_K)$-modules over $L$}}.
\end{equation*}
The subcategory of crystalline $(\varphi,\Gamma_K)$-modules is equivalent to the full subcategory of filtered $\varphi$-modules over $L$. The \'etale $(\varphi,\Gamma_K)$-modules (the Galois representations, following Theorem \ref{thm:galois-reps}) correspond to the weakly-admissible modules on the right-hand side.

Suppose that $D$ is crystalline. The $L\tensor_{\Q_p} K$-module $D_{\cris}(D)_K := D_{\cris}(D)\tensor_{F} K$ is equipped with an exhaustive and separated decreasing filtration $\Fil^{\bullet}D_{\cris}(D)_K$. We denote by $\HT_\tau(D)$ the multi-set of integers such that the induced filtration on the $L$-vector space $D_{\cris}(D)_{K,{\tau}} = e_\tau D_{\cris}(D)_K$ has jumps given with multiplicity by $\HT_\tau(D)$. It is easy to see that if $\delta: K^\x \goto L^\x$ is a character such that $\cal R_L(\delta)$ is crystalline then $\HT_\tau(\cal R_L(\delta)) = \wt_\tau(\delta)$. For example, $D_{\cris}(\cal R_L) = F\tensor_{\Q_p} L$ with the trivial $\varphi^{f_K}$-action and for all $\tau \in \Sigma_K$ we have $\HT_\tau(\cal R_L) = 0$. 

If $\tau: K \hookrightarrow L$ is an embedding then we denote the corresponding character $K^\x \goto L^\x$ by $z_\tau$. It happens that $D_{\cris}(\cal R_L(z_\tau))$ is a filtered $\varphi$-module with trivial $\varphi^{f_K}$-action and the Hodge--Tate filtration has weights
\begin{equation*}
\HT_\sigma(\cal R_L(z_\tau)) = \begin{cases}
-1 & \text{if $\sigma = \tau$}\\
0 & \text{if $\sigma \neq \tau$}
\end{cases}.
\end{equation*}
In particular $D_{\cris}(\cal R_L(z_\tau)) \ci D_{\cris}(\cal R_L)$ as filtered $\varphi$-modules. Thus $\cal R_L(z_\tau) = t_\tau \cal R_L$ for some $t_\tau$, uniquely determined up to unit in $\cal R_L$. 
\begin{proposition}\label{prop:rank-one-subs}
Every $(\varphi,\Gamma_K)$-submodule of $\cal R_L$ is of the form $\left(\prod_{\tau\in\Sigma_K} t_\tau^{r_\tau}\right)\cal R_L$ for some collection of non-negative integers $r_\tau \geq 0$.
\end{proposition}
\begin{proof}
See \cite[Corollary 6.2.9]{KedlayaPottharstXiao-Finiteness}.
\end{proof}

It's easy to see $t_\tau\cal R_L$ and $t_\sigma \cal R_L$ are maximally coprime if $\sigma \neq \tau$. Indeed, $D = t_\tau\cal R_L + t_\sigma \cal R_L$ is $(\varphi,\Gamma_K)$-submodule of $\cal R_L$ whose Hodge--Tate weights, computed by passing to $D_{\pst}(D)$, are zero (for each $\tau \in \Sigma_K$). The element $t = \prod_{\tau} t_\tau$ is, up to a unit, the ubiquitous $t$ which plays the role of the $p$-adic $2\pi i$ in all of $p$-adic Hodge theory. The $(\varphi,\Gamma_K)$-submodule $t\cal R_L$ is crystalline, its $\tau$-Hodge--Tate weight is $-1$ for each $\tau$ and $\varphi(t) = pt$.

\subsection{Galois cohomology}\label{subsec:galois-cohomology}
Suppose that $A$ is an affinoid algebra and that $Q$ is a generalized $(\varphi,\Gamma_K)$-module over $A$. Let $\Delta_K \ci \Gamma_K$ be the $p$-torsion subgroup (which only exists if $p=2$) and choose a topological generator $\gamma_0 \in \Gamma_K/\Delta_K$. One then defines the Herr complex \cite{Herr-Cohomology} as the three term complex $C^{\bullet}_{\gamma_0}(D)$
\begin{equation*}
\xymatrix{
Q^{\Delta_K} \ar[rrrr]^-{x \mapsto \left((\varphi-1)x,(\gamma_0-1)x\right)} &&&&  (Q^{\Delta_K})^{\dsum 2}  \ar[rrrr]^-{(y,z) \mapsto (\gamma_0-1)y - (\varphi-1)z} &&&&  Q^{\Delta_K}.
}
\end{equation*}
The Galois cohomology groups $H^{\bullet}(Q)$ of $Q$ are defined to be the cohomology groups of the complex $C^{\bullet}_{\gamma_0}(Q)$. The complexes depend on the choice of $\gamma_0$ up to canonical quasi-isomorphism, so the cohomology is well-defined.

When $A$ is finite over $L$, the Galois cohomology is finite-dimensional and satisfies an Euler--Poincar\'e formula \cite{Liu-CohomologyDuality}
\begin{equation*}
\sum_{i=0}^2 (-1) \dim_L H^i(Q) = -(K:\Q_p)\rank_{\cal R_L} Q,
\end{equation*}
where the rank of a torsion module must be suitably interpreted (see Section \ref{subsec:torsion}). The finiteness is true over general affinoid $L$-algebras $A$ in the case of $(\varphi,\Gamma_K)$-modules by \cite{KedlayaPottharstXiao-Finiteness}. We will review that result in Section  \ref{sec:finiteness}.

But now let us review the dimensions of the cohomology of rank one $(\varphi,\Gamma_K)$-modules over a field. To shorten notation, if $\delta: K^\x \goto L^\x$ is a continuous character we denote $H^{\bullet}(\delta) := H^{\bullet}(\cal R_A(\delta))$. We let $\hat T(L)$ be the space of continuous characters $\delta:K^\x \goto L^\x$. Define two special subsets of $\hat T(L)$ by
\begin{align*}
\hat T(L)^+ &= \set{\delta: K^\x \goto L^\x \st \delta = \prod_{\tau \in \Sigma_K} z_\tau^{r_\tau} \text{ with $r_\tau \leq 0$ for each $\tau$}}, \text{ and}\\
\hat T(L)^- &= \set{\delta: K^\x \goto L^\x \st \delta = \abs{N_{K/\Q_p}} \prod_{\tau \in \Sigma_K} z_\tau^{r_\tau} \text{ with $r_\tau \geq 1$ for each $\tau$}}.
\end{align*}
The elements of $\hat T(L)$ which are not in $\hat T(L)^+$ or $\hat T(L)^-$ are called generic characters.
\begin{proposition}\label{prop:cohomology-calcs}
Let $\delta: K^\x \goto L^\x$ be a continuous character. Then
\begin{align*}
\dim_{L} H^0(\delta) &= \begin{cases}
1 & \text{if $\delta \in \hat T(L)^+$};\\
0 & \text{otherwise.}
\end{cases}\\
\dim_{L} H^1(\delta) &= \begin{cases}
2 & \text{if $\delta \in  \hat T(L)^+ \union \hat T(L)^-$};\\
1 & \text{otherwise}.
\end{cases}\\
\dim_{L} H^2(\delta) &= \begin{cases}
1 & \text{if $\delta \in \hat T(L)^-$};\\
0 & \text{otherwise.}
\end{cases}
\end{align*}
\end{proposition}
\begin{proof}
See \cite[Proposition 6.2.8]{KedlayaPottharstXiao-Finiteness}) (or \cite[Section 2.3]{Nakamura-BPairs}).
\end{proof}

Let's finish this subsection with a definition.

\begin{definition}
Suppose that $\delta, \delta' \in \hat T(L)$. 
\begin{enumerate}
\item We say that $\delta$ and $\delta'$ are homothetic if there exists integers $(r_\tau)_{\tau \in \Sigma_K}$ such that $\delta = \delta' \prod_{\tau} z_\tau^{r_\tau}$.
\item We say that $\delta$ and $\delta'$ are generic up to homothety if $\delta'\delta^{-1}\prod_\tau z_\tau^{r_\tau}$ is generic for every tuple $(r_\tau)_{\tau \in \Sigma_K}$ of integers.
\end{enumerate}
\end{definition}
To put the previous definition in context, if $\delta$ and $\delta'$ are homothetic then following Proposition \ref{prop:cohomology-calcs} we can find an integer $r$ such that $t^r\cal R_L(\delta) \inject \cal R_L(\delta')$. Thus $\cal R_L(\delta)$ and $\cal R_L(\delta')$ are in a sense commensurable. If $\delta$ and $\delta'$ are generic up to homothety then $H^2(\eta) = H^0(\eta) = (0)$ for all characters $\eta$ homothetic to $\delta'\delta^{-1}$.

\subsection{Torsion $(\varphi,\Gamma_K)$-modules}\label{subsec:torsion}
Notice that
\begin{equation*}
\cal R_L = \cal R \tensor_{\Q_p} L = \cal R \tensor_F (F\tensor_{\Q_p} L) = \prod_{\eta \in \Gal(F/\Q_p)} \cal R \tensor_F L.
\end{equation*}
By \cite[Proposition 4.12]{Berger-Representationp-adique}, each term in the product is an adequate B\'ezout domain. In particular, finitely generated $\cal R_L$ modules are free if and only if they are torsion free (with respect to the total ring of divisors) and there is a robust theory of elementary divisors over $\cal R_L$. As a consequence, a generalized $(\varphi,\Gamma_K)$-module over $L$ is a $(\varphi,\Gamma_K)$-module if and only if it is torsion-free as an $\cal R_L$-module.

The element $t \in \cal R_L$ is an example of a non-zero divisor. If $S$ is a generalized $(\varphi,\Gamma_K)$-module we let $S[t^\infty]$ denote the $t$-power torsion submodule. Since $t \in \cal R$ is an eigenvector for $\varphi$ and $\Gamma_K$, $S[t^\infty]$ is a $(\varphi,\Gamma_K)$-submodule.

\begin{definition}
Let $A$ be an $L$-affinoid algebra. A torsion $(\varphi,\Gamma_K)$-module over $A$ is a generalized $(\varphi,\Gamma_K)$-module $S$ over $A$ such that $S[t^\infty] = S$. We say that $S$ is pure if either $S = 0$ or if $S$ is free over $\cal R_A/\left(\prod_{\tau} t_\tau^{r_\tau}\right)$ for some collection of integers $r_\tau \geq 0$, not all of which are zero.
\end{definition}

The typical example of a pure torsion $(\varphi,\Gamma_K)$-module is $\cal R_L/(\prod_{\tau} t_\tau^{r_\tau})\cal R_L(\delta)$ for some continuous character $\delta:K^\x \goto L^\x$ and non-negative integers $r_\tau$.

\begin{lemma}\label{lemma:torsion-free}
A generalized $(\varphi,\Gamma_K)$-module over $L$ is a $(\varphi,\Gamma_K)$-module if and only if it is $t$-torsion free. Any torsion $(\varphi,\Gamma_K)$-module over $L$ is a successive extension of pure torsion $(\varphi,\Gamma_K)$-modules.
\end{lemma}
\begin{proof}
The lemma is proven in the case $K = \Q_p$ in \cite[Proposition 4.1]{Liu-CohomologyDuality}. The proof in this case is the same, the main point being Proposition \ref{prop:rank-one-subs}. We reproduce it for convenience.

Let $D$ be a generalized $(\varphi,\Gamma_K)$-module. If it is a bona fide $(\varphi,\Gamma_K)$-module then it obviously cannot have $t$-torsion. Now suppose that $D$ is $t$-torsion free, and to show that $D$ is a bona fide $(\varphi,\Gamma_K)$-module it suffices to show that it is torsion free.

Since $D$ is finitely generated as a $\cal R_L$-module, and $\cal R_L$ is a B\'ezout domain, we have the theory of elementary divisors. Thus there exists a finite number of elements $d_1,\dotsc,d_m \in D$ which generate $D$ over $\cal R_L$ such that the ideals $f_i \cal R_L := \Ann_{\cal R_L}(d_i)$ are principal and $f_1\cal R_L \supset f_2\cal R_L \supset\dotsb \supset f_m \cal R_L$. The ideals $\set{f_i\cal R_L}$ are then uniquely determined by this property. We claim that each non-zero $f_i$ is a unit, which implies that $D$ is torsion free. Without loss of generality we can assume that $f_m \neq 0$ and show that $f_m$ is a unit.

We will first show  that each ideal $f_i\cal R_L$ is a $(\varphi,\Gamma_K)$-submodule of $\cal R_L$. If $\gamma \in \Gamma_K$ then it is easy to see that $\set{\gamma(d_i)}$ also generates $D$ as a $\cal R_L$-module. Furthermore, $\Ann_{\cal R_L}(\gamma(d_i)) \ci \Ann_{\cal R_L}(\gamma(d_{i-1}))$ for $1 < i \leq m$. Thus by the uniqueness in the theory of elementary divisors, we have $\gamma(\Ann_{\cal R_L}(d_i)) = \Ann_{\cal R_L}{\gamma(d_i)} = \Ann_{\cal R_L}(d_i)$ for each $i$. This shows that each ideal $f_i\cal R_L$ is $\Gamma_K$-stable. 

On the other hand, one may also check that the elementary divisors for $\varphi^{\ast}D$ are the ideals $\Ann_{\cal R_L}(d_i\tensor 1)$ (the elements $d_i \tensor 1$ written as tensors in  $\varphi^{\ast}D = D\tensor_{\cal R_L,\varphi} \cal R_L$). Since $\varphi^{\ast}D\simeq D$ the elementary divisors for both $\varphi^{\ast}D$ and $D$ are the same, hence $\Ann_{\cal R_L}(d_i\otimes 1) = \Ann_{\cal R_L}(d_i)$. This required only having an abstract isomorphism between $\varphi^{\ast}D$ and $D$. On the other hand, the isomorphism $\varphi^{\ast}D \goto D$ is defined explicitly by sending $d_i\otimes 1$ to $\varphi(d_i)$, which implies that $\Ann_{\cal R_L}(\varphi(d_i))=\Ann_{\cal R_L}(d_i\tensor 1)$. Finally, we see that $\varphi(\Ann_{\cal R_L}(d_i)) \ci \Ann_{\cal R_L}(\varphi(d_i)) = \Ann_{\cal R_L}(d_i)$. Thus each ideal $f_i\cal R_L$ is $\varphi$-stable as well.

We now finish the proof. By Proposition \ref{prop:rank-one-subs}, if $f_m$ is not unit then there exists non-negative integers $r_{\tau} \geq 0$, not all zero, such that $\left(\prod_{\tau \in \Sigma_K} t_\tau^{r_\tau}\right)d_m = 0$. Since $\left(\prod_{\tau \in \Sigma_K} t_\tau^{r_\tau}\right)$ is a divisor of $t^r$ for $r$ large, we conclude that $d_m \in D[t^\infty]$. Since $D$ is $t$-torsion free, we conclude that $d_m = 0$, which contradicts the choice of the elements $\set{d_i}$. The calculation also clearly shows that a torsion $(\varphi,\Gamma_K)$-module is a successive extension of pure torsion $(\varphi,\Gamma_K)$-modules. Indeed, if $Q$ is torsion then $0 = f_mQ \ci f_{m-1}Q \ci \dotsb \ci f_1Q \ci Q$ is a filtration whose successive quotients $f_iQ/f_{i+1}Q$ are pure torsion, since they are free over $f_i\cal R_{L}/f_{i+1}\cal R_L$ (compare with \cite[Proposition 4.1]{Liu-CohomologyDuality}).
\end{proof}

If $S$ is a torsion module then the Euler--Poincar\'e formula \cite[Theorem 4.7]{Liu-CohomologyDuality} says
\begin{equation*}
\dim_L H^0(S) = \dim_L H^1(S) \text{ and } \dim_L H^2(S) = 0.
\end{equation*}
By Lemma \ref{lemma:torsion-free}, the cohomology of torsion $(\varphi,\Gamma_K)$-modules reduces to the cohomology of pure torsion $(\varphi,\Gamma_K)$-modules and that is explained by the following calculation.
\begin{proposition}\label{prop:torsion-cohomology}
Let $\tau \in \Sigma_K$. Then for each $i = 0,1$ we have
\begin{equation*}
\dim_L H^i((\cal R_L/t_\tau^{r_\tau})(\delta)) = \begin{cases}
1 & \text{if $\wt_\tau(\delta) \in \set{0,1,\dotsc,r_\tau-1}$}\\
0 & \text{otherwise.}
\end{cases}
\end{equation*}
\end{proposition}
\begin{proof}
The case of $i=1$ and $i=0$ are equivalent by the Euler--Poincar\'e formula for torsion modules. The computation of the cohomology for $i = 0$ is given by \cite[Lemma 2.16]{Nakamura-BPairs}. For a proof in the language of $(\varphi,\Gamma_K)$-modules, at least when $K = \Q_p$, see \cite[Proposition 2.18]{Colmez-Trianguline} (warning: Colmez uses a different convention for weights).
\end{proof}

\section{Parabolizations and triangulations}\label{sec:triangulations}
Triangulations of $(\varphi,\Gamma_K)$-modules arose following Colmez's work on the $p$-adic local Langlands for $\GL_2(\Q_p)$ (see \cite{Colmez-Trianguline}). In this section we have two goals. First, we will recall parabolizations of $(\varphi,\Gamma_K)$-modules, a more general notion due to Chenevier \cite{Chenevier-InfiniteFern}, and the definition of critical and non-critical triangulations. Our second goal is to extend the definition of triangulation in a reasonable way to the category of generalized $(\varphi,\Gamma_K)$-modules. We discuss the latter notion only in the case where the coefficients are a field.

\subsection{Parabolizations of $(\varphi,\Gamma_K)$-modules}\label{subsec:parabolizations}
Let $A$ be an affinoid $L$-algebra.

\begin{definition}\label{defn:triangulation}
If $D$ is a $(\varphi,\Gamma_K)$-module over $A$ then a parabolization $P_{\bullet}$ of $D$ (of length $s$) is a filtration
\begin{equation*}
0 = P_0 \sci P_1 \sci \dotsb \sci P_{s-1} \sci P_s = D
\end{equation*}
such that each 
\begin{itemize}
\item $P_i$ is a $(\varphi,\Gamma_K)$-module and 
\item for each $i=1,\dotsc,s$ we have that $P_i/P_{i-1}$ is a $(\varphi,\Gamma_K)$-module over $A$ which is a $\cal R_A$-module direct summand of $P_i$.
\end{itemize}
If $P_{\bullet}$ is a parabolization of the maximal length $s = \rank_{\cal R_A} D$ then we say that $P_{\bullet}$ is a triangulation. We say $D$ is trianguline if, after possibly extending the coefficient field $L$, there exists a triangulation of $D$. 
\end{definition}

If $P_{\bullet}$ is a triangulation of a $(\varphi,\Gamma_K)$-module $D$ of rank $d$ then each quotient $P_i/P_{i-1}$ is of the form $\cal R_A(\delta_i)$ for some continuous character $\delta_i: K^\x \goto A^\x$, at least locally on $X = \Sp(A)$ \cite[Theorem 6.2.14]{KedlayaPottharstXiao-Finiteness}. We call the $d$-tuple $(\delta_i)_{i=1}^{d}$ the ordered parameter of the triangulation $P_{\bullet}$ and we say that $D$ is trianguline with ordered parameter $(\delta_i)_{i=1}^d$. If $D$ is trianguline with an ordered parameter $(\delta_i)_{i=1}^d$ then $\HT_\tau(D) = \set{\wt_\tau(\delta_i)}_{i=1}^d$.

For the rest of this section we will take $A$ to be the field $L$ itself. Crystalline $(\varphi,\Gamma_K)$-modules over $\cal R_L$ provide examples of triangulations. For that we have the notion of a refinement, following \cite[Definition 5.29]{Liu-Triangulations}.

\begin{definition}
If $D$ is a crystalline $(\varphi,\Gamma_K)$-module of rank $d$ over $L$ then a partial refinement $R_{\bullet}$ of $D$ is the choice of a $\varphi$-stable $L\tensor_{\Q_p} F$-linear filtration
\begin{equation*}
0 = R_0 \sci R_1 \sci \dotsb \sci R_s = D_{\cris}(D).
\end{equation*}
whose successive quotients are free $L\tensor_{\Q_p} F$-modules. In the case that $s = d$ we call $R_{\bullet}$ a refinement.
\end{definition}

Suppose that $R_{\bullet}$ is a refinement of a crystalline $(\varphi,\Gamma_K)$-module. Then each of the quotients $R_i/R_{i-1}$ is a rank one $L\tensor_{\Q_p} F$-module equipped with a linear operator $\varphi^{f_K}$. We denote by $\phi_i \in L^\x$ the eigenvalue of $\varphi^{f_K}$ appearing in $R_i/R_{i-1}$. Furthermore, $D_{\cris}(D)_K$ is an $L\tensor_{\Q_p} K$-vector space equipped with its Hodge filtration $\Fil^{\bullet}D_{\cris}(D)_K$. Each of the $\varphi$-stable subspaces $(R_i)_K$ has an induced Hodge filtration. We define, for each $\tau \in \Sigma_K$ and $i=1,\dotsc,d$ an integer $s_{i,\tau}$ so that $\set{s_{1,\tau},\dotsc,s_{i,\tau}}$ are the $\tau$-Hodge--Tate weights appearing in $(R_i)_{K,\tau}$. In summary, triangulations and refinements have the following invariants
\begin{align*}
\text{A triangulation $P_{\bullet}$} \rightsquigarrow& \text{ the ordered  parameter $(\delta_1,\dotsc,\delta_n)$}.\\
\text{A refinement $R_{\bullet}$} \rightsquigarrow &\text{\parbox{8cm}{\centering the ordering of $\varphi^{f_K}$-eigenvalues $(\phi_1,\dotsc,\phi_n)$ and the $\tau$-Hodge--Tate weights $(s_{1,\tau},\dotsc,s_{d,\tau})_{\tau}$}}.
\end{align*}
Note that if $P_{\bullet}$ is a triangulation of a crystalline $(\varphi,\Gamma_K)$-module $D$ then each step $P_i$ is a crystalline $(\varphi,\Gamma_K)$-module as well.
\begin{proposition}\label{prop:refinements-triangulations}
Let $D$ be a crystalline $(\varphi,\Gamma_K)$-module over $L$ all of whose $\varphi^{f_K}$-eigenvalues lie in $L^\x$.
\begin{enumerate}
\item  Then $P \mapsto D_{\cris}(P)$ induces bijections
\begin{align*}
\set{\text{parabolizations of $D$}} &\longleftrightarrow \set{\text{partial refinements of $D$}}, \text{and}\\
\set{\text{triangulations of $D$}} &\longleftrightarrow \set{\text{refinements of $D$}}.
\end{align*}
\item If $P_{\bullet}$ is a triangulation with ordered parameter $(\delta_1,\dotsc,\delta_n)$ then the orderings associated to $D_{\cris}(P_{\bullet})$ are given by
\begin{align*}
(s_{1,\tau},\dotsc,s_{d,\tau}) &= (\wt_\tau(\delta_1),\dotsc,\wt_\tau(\delta_n)), \text{and}\\
(\phi_1,\dotsc,\phi_n) &= \left(\delta_1(\varpi_K)\prod_{\tau \in \Sigma_K} \tau(\varpi_K)^{\wt_\tau(\delta_1)}, \dotsc, \delta_d(\varpi_K)\prod_{\tau \in \Sigma_K}\tau(\varpi_K)^{\wt_\tau(\delta_d)}\right),
\end{align*}
for some (or, any) choice of uniformizer $\varpi_K \in K^\x$.

\item If $R_{\bullet}$ is a refinement with orderings $(\phi_1,\dotsc,\phi_n)$ and $(s_{1,\tau},\dotsc,s_{d,\tau})_{\tau \in \Sigma_K}$ then the parameter $(\delta_1,\dotsc,\delta_n)$ of the corresponding triangulation $P_{\bullet}$ is given by
\begin{equation*}
\delta_i(z) = (\prod_{\tau \in \Sigma_K} z_\tau^{-s_{i,\tau}})\unr_{\varpi_K}(\phi_i)(z)
\end{equation*}
for some (or, any) choice of uniformizer $\varpi_K \in K^\x$.
\end{enumerate}
\end{proposition}
\begin{proof}
The first part follows from Berger's dictionary \cite{Berger-EquationsDifferentielles} between potentially semistable $(\varphi,\Gamma_K)$-modules and filtered $(\varphi,N,G_K)$-modules. The second two parts are easy inductions from the rank one case. In the case that $K = \Q_p$ a longer discussion can be found in \cite[Proposition 2.4.1]{BellaicheChenevier-Book}.
\end{proof}

\subsection{Critical and non-critical triangulations}\label{subsect:critical-noncritical}
For this subsection we work with a fixed crystalline $(\varphi,\Gamma_K)$-module $D$ over the field $L$.

\begin{definition}
Suppose that $P \ci D$ is a saturated $(\varphi,\Gamma_K)$-submodule. If $\tau \in \Sigma_K$ then we say that $P$ is $\tau$-non-critical if there exist an integer $k_\tau$ such that 
\begin{equation*}
D_{\cris}(P)_{K,\tau} \dsum \Fil^{k_\tau}D_{\cris}(D)_{K,\tau} = D_{\cris}(D)_{K,\tau}.
\end{equation*}
We say $P$ is $\tau$-critical otherwise. Finally, $P$ is called non-critical if $P$ is $\tau$-non-critical for each $\tau \in \Sigma_K$ and $P$ is critical if there exists a $\tau \in \Sigma_K$ such that $P$ is $\tau$-critical.
\end{definition}

The definition is only given for crystalline $(\varphi,\Gamma_K)$-modules as it relies on the correspondence Proposition \ref{prop:refinements-triangulations}. A more general definition will be given later which applies to certain $p$-adic limits of crystalline $(\varphi,\Gamma_K)$-modules (see Definition \ref{defn:non-critical-two}).

In the case of regular Hodge--Tate weights, we have a convenient way to check whether or not a saturated $(\varphi,\Gamma_K)$-submodule is critical.

\begin{lemma}\label{lem:equivalent-noncritical}
Let $\tau \in \Sigma_K$. Suppose that $D$ is a crystalline $(\varphi,\Gamma_K)$-module with regular $\tau$-Hodge--Tate weights $k_{1,\tau} < \dotsb < k_{d,\tau}$. Let $P \ci D$ be a saturated $(\varphi,\Gamma_K)$-submodule of rank $i\leq d$. The following are equivalent:
\begin{enumerate}
\item $P$ is $\tau$-non-critical;
\item $D_{\cris}(P)_{K,\tau} \dsum \Fil^{k_{i+1,\tau}}D_{\cris}(D)_{K,\tau} = D_{\cris}(D)_{K,\tau}$;
\item $\HT_\tau(P) = \set{k_{1,\tau},\dotsc,k_{i,\tau}}$;
\item $\det P \ci \wedge^i D$ is $\tau$-non-critical.
\end{enumerate}
\end{lemma}
\begin{proof}
First, (b) implies (a) by definition. Second, (b) and (c) are easily equivalent. Now suppose that $P$ is non-critical and choose an integer $k_\tau$ such that $D_{\cris}(P)_{K,\tau} \dsum \Fil^{k_\tau}D_{\cris}(D)_{K,\tau} = D_{\cris}(D)_{K,\tau}$. Since $P$ is crystalline, $d-i = \dim_{L_\tau} \Fil^{k_\tau}D_{\cris}(D)_{K,\tau}$. Since the Hodge--Tate weights are all distinct we conclude $\Fil^{k_\tau}D_{\cris}(D)_{K,\tau} = \Fil^{k_{i+1,\tau}}D_{\cris}(D)_{K,\tau}$. This shows (a) implies (b).

It remains to show that (c) and (d) are equivalent. Since $D$ has distinct Hodge--Tate weights, the unique lowest weight of $\wedge^i D$ is $k_{1,\tau}+\dotsb+k_{i,\tau}$. The next highest weight is $k_{1,\tau}+\dotsb+k_{i-1,\tau} + k_{i+1,\tau}$. Thus (c) is true if and only if 
\begin{equation*}
D_{\cris}(\det P)_{K,\tau} \dsum \Fil^{k_{1,\tau}+\dotsb+k_{i-1,\tau} + k_{i+1,\tau}}D_{\cris}(\wedge^i D)_{K,\tau} = D_{\cris}(\wedge^i D)_{K,\tau},
\end{equation*}
which is (d).
\end{proof}

At this point, one could define what it means for a triangulation to be non-critical. More generally, for each parabolization $P_{\bullet}$ of a crystalline $(\varphi,\Gamma_K)$-module we define a subparabolization $P^{\nc}_{\bullet} \ci P_{\bullet}$ for which every step is non-critical.

\begin{definition}
Let $P_{\bullet}$ be a parabolization of a crystalline $(\varphi,\Gamma_K)$-module $D$ of rank $d$. Let 
\begin{equation*}
I^{\nc} = \set{i \st \text{$P_i$ is non-critical}} = \set{0 = i_0 < i_1 < \dotsb < i_r = d}.
\end{equation*}
The maximal non-critical parabolization $P^{\nc}_{\bullet}$ is the filtration
\begin{equation*}
P^{\nc}_{\bullet}: 0 = P_{i_0} \sci P_{i_1} \sci \dotsb \sci P_{i_r} = D.
\end{equation*}
We say that $P_{\bullet}$ is non-critical if $P_{\bullet}^{\nc} = P^{\nc}$, and critical otherwise.
\end{definition}
Notice that, as suggested by our notation, $I^{\nc} \neq \es$ and $i_r = d$, since $D$ itself is always a non-critical $(\varphi,\Gamma_K)$-submodule of itself. In the case where $D$ has regular weights, Lemma \ref{lem:equivalent-noncritical} shows that $P^{\nc}_{\bullet}$ is the unique subparabolization of $P_{\bullet}$ consisting of the steps whose Hodge--Tate weights are as low as possible. Furthermore, it is easy to check that $(P_{\bullet}^{\nc})^{\nc} = P_{\bullet}^{\nc}$, hence the use of the word ``maximal''. Let us end this subsection with a brief example.

\begin{example}
Suppose now that $K = \Q_p$ and that $D$ is a rank two crystalline $(\varphi,\Gamma_{\Q_p})$-module over $L$, with Hodge--Tate weights $k_1 < k_2$ and distinct crystalline eigenvalues $\phi,\phi' \in L^\x$. Since $\phi\neq \phi'$ we assume without loss of generality that $D_{\cris}(D)^{\varphi = \phi} \intersect \Fil^{k_2}D_{\cris}(D) = (0)$. Thus there is always a non-critical triangulation $\cal R_L(z^{-k_1}\unr(\phi)) \ci D$. The ordered parameter is $(z^{-k_1}\unr(\phi), z^{-k_2}\unr(\phi'))$.

On the other hand, one can use Propositions \ref{prop:refinements-triangulations} and \ref{prop:cohomology-calcs} to show that $D$ is split if and only if $D_{\cris}(D)^{\varphi = \phi'} = \Fil^{k_2}D_{\cris}(D)$ (if $D$ is \'etale, the same statement follows from the weak admissibility of the filtered $\varphi$-module $D_{\cris}(D)$). Thus the triangulation corresponding to the ordering $(\phi',\phi)$ is given by
\begin{equation*}
\begin{cases}
\cal R_L(z^{-k_1}\unr(\phi')) \ci D & \text{if $D$ is non-split,}\\
\cal R_L(z^{-k_2}\unr(\phi')) \ci D & \text{if $D$ is split.}
\end{cases}
\end{equation*}
This triangulation is critical if and only if $D$ is split.
\end{example}

\subsection{Generalized triangulations}\label{subsect:generalized-triangulations}
Recall that a pure torsion $(\varphi,\Gamma_K)$-module over $L$ is a generalized $(\varphi,\Gamma_K)$-module that is either zero or free over $\cal R_L/\prod_\tau t_\tau^{r_\tau}$ for some collection $(r_\tau)_{\tau}$ of non-negative integers, not all zero. 

\begin{definition}
We say a pure torsion $(\varphi,\Gamma_K)$-module $Q$ is of character type if either $Q = 0$ or there exists a continuous character $\delta:K^\x \goto L^\x$ and a collection of non-negative integers $(r_\tau)_{\tau}$, not all zero, such that $Q \simeq \coker((\prod_\tau t_\tau^{r_\tau})\cal R_L(\delta) \goto \cal R_L(\delta))$.
\end{definition}
If $Q$ is non-zero and pure torsion of character type then we refer to the $\Sigma_K$-tuple $(r_\tau)_{\tau \in \Sigma_K}$ as the torsion exponents of $Q$ and $(\wt_\tau(\delta))_{\tau \in \Sigma_K}$ as the torsion weights of $Q$. The zero module $(0)$ has, by definition, torsion exponents $(0,\dotsc,0)$ and torsion weights $(w_\tau)_{\tau \in \Sigma_K}$ for any collection of integers $w_\tau$. By Proposition \ref{prop:torsion-cohomology}, these invariants, taken together, completely classify $Q$ among pure torsion $(\varphi,\Gamma_K)$-modules of character type.

\begin{definition}
Let $Q$ be a generalized $(\varphi,\Gamma_K)$-module over $\cal R_L$. A generalized triangulation of $Q$ is a filtration $Q_{\bullet}$
\begin{equation*}
Q_{\bullet}: 0 = Q_0 \ci Q_1 \ci Q_2 \ci \dotsb \ci Q_{d-1} \ci Q_d = Q
\end{equation*}
such that for $1 \leq i \leq d$, $Q_i/Q_{i-1}$ is either a rank one $(\varphi,\Gamma_K)$-module or a pure torsion $(\varphi,\Gamma_K)$-module of character type and in either case $Q_i/Q_{i-1}$ is a direct summand of $Q_i$ as a $\cal R_L$-module. We say that $Q$ is triangulated if it is equipped with a triangulation and trianguline if it may be triangulated, after possibly extending scalars.
\end{definition}

Note that we allow for consecutive steps $Q_i \ci Q_{i+1}$ to be equal, since $(0)$ is a pure torsion $(\varphi,\Gamma_K)$-module of character type under our definition. This has two consequences. First, even if $Q$ is a bona fide $(\varphi,\Gamma_K)$-module then a generalized triangulation is not a triangulation in the sense of Section \ref{subsec:parabolizations}. We will deal with this ambiguity in the definition of standard triangulation below. Second, since we can always repeat steps in a generalized triangulation, the length of a generalized triangulation $Q_{\bullet}$ depends on $Q_{\bullet}$; it is not intrinsic to $Q$, unlike lengths of triangulations of bona fide $(\varphi,\Gamma_K)$-modules. This remains true even if a generalized triangulation $Q_{\bullet}$ is strictly increasing: the length still cannot be read off from $Q$ since $\cal R_L/t_\tau t_\sigma \simeq \cal R_L/t_\tau \oplus \cal R_L/t_\sigma$ if $\sigma \neq \tau$.

Note that if $Q$ is a generalized $(\varphi,\Gamma_K)$-module then the torsion submodule $Q_{\tor} = Q[t^\infty] \ci Q$ is $(\varphi,\Gamma_K)$-stable and an $\cal R_L$-module summand. The quotient $Q/Q_{\tor}$ is a bona fide $(\varphi,\Gamma_K)$-module whose rank depends only on $Q$. We isolate those generalized triangulations which  appear in practice.
\begin{definition}
If $Q$ is a generalized $(\varphi,\Gamma)$-module and $Q_{\bullet}$ is a generalized triangulation of $Q$ then $Q_{\bullet}$ is called a standard triangulation if there exists a $0\leq i \leq d$ such that $Q_i = Q_{\tor}$. The integer $i$ is called the torsion length of $Q_{\bullet}$ and the integer $d-i$ is called the free length of $Q$.
\end{definition} 
In the definition, the torsion length $i$ depends on $Q_{\bullet}$ whereas the free length $d-i$ depends only on $Q$ (since it is equal to $\rank_{\cal R_L} Q/Q_{\tor}$). And now a standard triangulation is closer to a triangulation in the case where $Q$ is a bona fide $(\varphi,\Gamma_K)$-module. Indeed, in that case $Q_{\tor} = (0)$ and so a standard triangulation is of the form
\begin{equation*}
0 = 0 = \dotsb = 0 = Q_i \sci {Q}_{i+1} \sci \dotsb \sci Q_d = Q
\end{equation*}
where $P_j := Q_{i+j}$ defines a triangulation of $Q$ as in Definition \ref{defn:triangulation}.

A standard triangulation of a generalized $(\varphi,\Gamma_K)$-module has a number of invariants which we now detail. Suppose that $Q_{\bullet}$ is a standard triangulation. 
\begin{itemize}
\item The induced generalized triangulation on $Q/Q_{\tor}$ is an actual triangulation with ordered parameter $(\delta_{i+1},\dotsc,\delta_d)$ whose length is the free length of $Q$. 
\item The induced generalized triangulation on $Q_{\tor}$  also has invariants. For one, it has its length $i$. Second, if $1 \leq j \leq i$ then $Q_j/Q_{j-1}$ is pure torsion of character type and thus has exponents $(r_{j,\tau})_{\tau}$ and weights $(w_{j,\tau})_\tau$.
\end{itemize}
Note that it may happen that for some $\tau$, $r_{j,\tau} = 0$. For example, if there is an integer $j$ such that $Q_j = Q_{j-1}$ then the corresponding torsion exponents are $r_{j,\tau} = 0$ for all $\tau$. However, ranging over $j$ we can {\em a priori} predict the frequency at which this happens.

\begin{definition}
If $Q$ is a generalized $(\varphi,\Gamma_K)$-module then its $\tau$-torsion length is defined by
\begin{equation*}
\ell_\tau(Q) = \rank_{\cal R_L/t_\tau} Q_{\tor}/t_\tau.
\end{equation*}
\end{definition}

Fix a $\tau \in \Sigma_K$ and a standard triangulation $Q_{\bullet}$ of a generalized $(\varphi,\Gamma_K)$-module with torsion length $i$ and exponents $((r_{j,\tau})_{\tau})_{1\leq j \leq i}$. Since the successive quotients of a generalized triangulation are direct summands as $\cal R_L$-modules, it is easy to see that $\ell_\tau(Q) = \sizeof\set{j \st r_{j,\tau} \neq 0}$ and that $\ell_\tau(Q) \leq i$ for all $\tau$. To summarize the previous discussion we separate out the following definition.

\begin{definition}
If $Q_\bullet$ is a standard triangulation of a generalized $(\varphi,\Gamma_K)$-module $Q$ then, in the notation above,
\begin{itemize}
\item the torsion length is the unique integer $i\geq 1$ such that $Q_{\tor} = Q_i$,
\item the torsion exponents are $((r_{j,\tau})_{\tau})_{1\leq j \leq i}$, 
\item the torsion weights are $((w_{j,\tau})_\tau)_{1\leq j\leq i}$ and 
\item the free parameter is the ordered parameter $(\delta_j)_{j>i}$.
\end{itemize}
If we specify an element $\tau \in \Sigma_K$ then we refer to $(r_{j,\tau})_{1\leq j\leq i}$ and $(w_{j,\tau})_{1\leq j\leq i}$ as the $\tau$-torsion exponents and $\tau$-torsion weights.
\end{definition}

Finally, we finish this section with a result that explains how standard triangulations of generalized $(\varphi,\Gamma_K)$-modules are inherently more flexible than triangulations of bona fide $(\varphi,\Gamma_K)$-modules. Recall that we defined the notion of homothety among continuous characters of $K^\x$ at the end of Section \ref{subsec:galois-cohomology}.

\begin{proposition}\label{lemm:stupid-lemma}
Suppose that $Q$ is a generalized $(\varphi,\Gamma_K)$-module, $Q_{\bullet}$ is a standard triangulation of torsion length $i$ with torsion exponents $((r_{j,\tau})_\tau)_{1\leq j\leq i}$, torsion weights $((w_{j,\tau})_\tau)_{1\leq j\leq i}$ and free parameter $(\delta_j)_{j>i}$. Assume furthermore that $\delta_{i+1}$ is not homothetic to $\delta_j$ for $j > i+1$.

Then, for every $\Sigma_K$-tuple $(r_{i+1,\tau})_{\tau}$ of non-negative integers such that
\begin{align*}
j\leq i \implies w_{j,\tau} - \wt_\tau(\delta_{i+1}) \nin& \set{-r_{i+1,\tau},\dotsc, r_{j,\tau} - r_{i+1,\tau} - 1}
= \set{-r_{i+1,\tau} + m \st 0\leq m < r_{j,\tau}},
\end{align*}
there exists a unique (up to scalar) inclusion $\prod_{\tau} t_\tau^{r_{i+1,\tau}}\cal R_{L}(\delta_{i+1}) \inject Q$. Its cokernel is a generalized $(\varphi,\Gamma_K)$-module which is naturally equipped with a standard triangulation having invariants:
\begin{itemize}
\item torsion length $i+1$
\item torsion exponents $((r_{j,\tau})_\tau)_{j\leq i+1}$,
\item torsion weights $((w_{j,\tau})_\tau)_{j\leq i} \union (\wt_\tau(\delta_{i+1}))_\tau$ and
\item free parameter $(\delta_{j})_{j > i+1}$.
\end{itemize}
\end{proposition}
\begin{proof}
Let $\delta = \delta_{i+1}\prod_{\tau} z_\tau^{r_{i+1,\tau}}$. For each $j\leq i$ we choose a character $\delta_j$ whose $\tau$-weight is $w_{j,\tau}$ and so that $Q_{j}/Q_{j-1} \simeq \cal R_L(\delta_j)/\prod_{\tau} t_\tau^{r_{j,\tau}}$. We quickly calculate
\begin{equation*}
\wt_\tau(\delta_j\delta^{-1}) = w_{j,\tau} - \wt_{\tau}(\delta_{i+1}) + r_{i+1,\tau}.
\end{equation*}
By our assumptions, $\wt_\tau(\delta_j\delta^{-1}))  \nin \set{0,1,\dotsc,r_{j,\tau}-1}$ for each $\tau$. Using Proposition \ref{prop:torsion-cohomology} we see that 
\begin{equation*}
\Hom(\cal R_L(\delta),Q_j/Q_{j-1}) = H^0((\cal R_L/\prod_\tau t_\tau^{r_{j,\tau}})(\delta_j\delta^{-1})) = (0).
\end{equation*}
By induction on $1\leq j \leq i$ we see that $\Hom(\cal R_L(\delta),Q_i) = (0)$. Since $Q_i$ is torsion, the Euler--Poincar\'e formula for torsion modules implies $H^1(Q_i(\delta^{-1})) = (0)$ as well. We deduce from the long exact sequence in cohomology that the natural map $\Hom(\cal R_L(\delta),Q) \goto \Hom(\cal R_L(\delta),Q/Q_i)$ is an isomorphism.

On the other hand, $Q/Q_i$ is triangulated by a parameter $(\delta_{i+1},\dotsc)$ whose higher terms are not homothetic to $\delta_{i+1}$. From that we deduce that the inclusion $\Hom(\cal R_L(\delta),\cal R_L(\delta_{i+1})) \inject \Hom(\cal R_L(\delta),Q/Q_{i})$ is an isomorphism also. Putting the two calculations together, we see that
\begin{equation*}
\dim_L \Hom(\cal R_L(\delta),Q) = \dim_L \Hom(\cal R_L(\delta), \cal R_L(\delta_{i+1})) =1.
\end{equation*}
This shows that the morphism in the lemma exists and is unique up to a scalar.

But, the calculation shows more. We've shown in fact that any non-zero morphism $e: \cal R_L(\delta) \goto Q$ factors through $Q_{i+1}$ and that $e$ remains non-zero when mapped into the quotient $\cal R_L(\delta_{i+1})$ of $Q_{i+1}$. Since both $\cal R_L(\delta)$ and $\cal R_L(\delta_{i+1})$ are rank one, $e$ must be injective and it induces induces an exact sequence
\begin{equation}\label{eqn:ses-itmustsplit}
0 \goto Q_{i} \goto Q_{i+1}/\cal R_L(\delta) \goto \cal R_L(\delta_{i+1})/\cal R_{L}(\delta) \goto 0.
\end{equation}
Since $\cal R_L(\delta_{i+1})$ is a direct summand of $Q_{i+1}$ as a $\cal R_L$-module, the sequence \eqref{eqn:ses-itmustsplit} is also split as a sequence of $\cal R_L$-modules. This means that the standard triangulation $Q_{\bullet}$ on $Q$ induces a standard triangulation $Q_{\bullet}'$ on $Q/\cal R_L(\delta)$ whose successive quotients are given by
\begin{equation*}
Q'_{j}/Q'_{j-1} = \begin{cases}
Q_j/Q_{j-1} & \text{if $j \neq i+1$}\\
\cal R_L(\delta_{i+1})/\prod_\tau t_\tau^{r_{i+1,\tau}} & \text{if $j = i+1$}.
\end{cases}
\end{equation*}
The invariants are of the new standard triangulation are easily calculated from this.
\end{proof}
Note that one can always find infinitely many such integers $r_{i+1,\tau}$ which satisfy the hypotheses of the proposition. Also note that if $r_{i+1,\tau} = 0$ for all $\tau \in \Sigma_K$ then the standard triangulation we just produced will have two consecutive steps which are equal.

\section{Galois cohomology in families}
\label{sec:finiteness}

In this short section we expand on Section \ref{subsec:galois-cohomology}. In particular, we recall the main results of \cite{KedlayaPottharstXiao-Finiteness} and develop a simple cohomology and base change framework for generalized $(\varphi,\Gamma_K)$-modules. The framework will be applied in Sections \ref{sec:triangulated-families} and \ref{sec:variation}. Throughout this section we will let $A$ be a {\em reduced} affinoid $L$-algebra and $X = \Sp(A)$.

 Suppose that $N_{\bullet} = [ \dotsb \rightarrow N_1 \rightarrow N_0]$ is a complex of $A$-modules and $M$ is an $A$-module such that $\Tor_j^A(N_p, M) = (0)$ for each $j\geq 1$ and $p\geq 0$. Then the K\"unneth spectral sequence, see \cite[Theorem 5.6.4]{Weibel}, is a first quadrant spectral sequence
\begin{equation*}
E^2_{pq} = \Tor_p^A(H_q(N_{\bullet}), M) \Rightarrow H_{p+q}(N_{\bullet}\tensor_A M).
\end{equation*}

\begin{definition}
An $A$-module $Q$ is called nearly flat if $\Tor_j^A(Q,L(x)) = (0)$ for all $j\geq 1$ and $x \in \Sp(A)$.
\end{definition}

Recall that if $x \in X(\bar \Q_p)$ and $Q$ is an $A$-module then $Q_x$ denotes the fiber $Q\tensor_{A} L(x)$.

\begin{proposition}\label{prop:coh-kunneth}
If $Q$ is a nearly flat generalized $(\varphi,\Gamma_K)$-module over $A$ then there is a first quadrant spectral sequence
\begin{equation*}
\Tor_p^A(H^{2-q}(Q),L(x)) \Rightarrow H^{2-(p+q)}(Q_x)
\end{equation*}
which degenerates on the $E^3$-page.
\end{proposition}
\begin{proof}
We apply the K\"unneth spectral sequence to $M = L(x)$ and the three term Herr complex $C^{\bullet}_{\varphi,\gamma_K}(Q)$, after making the obvious shift between homology and cohomology. The hypothesis in the K\"unneth spectral sequence is valid since $Q$ is nearly flat and each term of $C^{\bullet}_{\varphi,\gamma_K}(Q)$ is a direct sum of finitely many copies of $Q$, if $p$ is odd. If $p=2$ then each term of $C_{\varphi,\gamma_K}^{\bullet}(Q)$ is actually a direct sum of finitely many copies of $Q^{\Delta_K}$, itself a direct summand of $Q$ (we thank the anonymous referee for this precision).

As an aid to the reader, let us explicitly write out the $E^2$-page of the spectral sequence
\begin{equation}
\xymatrix{
\vdots & \vdots & \vdots \\
0 & 0 & 0 & \dotsb\\
H^0(Q)\tensor_A L(x)& \Tor_1^A(H^0(Q),L(x)) & \Tor_2^A(H^0(Q),L(x)) & \dotsb\\
H^1(Q) \tensor_A L(x) & \Tor_1^A(H^1(Q),L(x)) & \Tor_2^A(H^1(Q),L(x)) \ar[ull] & \dotsb\\
H^2(Q) \tensor_A L(x) & \Tor_1^A(H^2(Q),L(x)) & \Tor_2^A(H^2(Q),L(x)) \ar[ull] & \dotsb
}
\end{equation}
The arrows drawn are the differentials. And now it is clear that the spectral sequence stabilizes on the $E^3$-page since the differentials there and afterwards are all zero.
\end{proof}

\begin{proposition}\label{prop:torsion-kunneth}
Suppose $f \in \cal R_L$ is not a zero divisor. If $Q$ is a nearly flat generalized $(\varphi,\Gamma_K)$-module over $A$ then for each $x \in X$:
\begin{enumerate}
\item There is a four term exact sequence
\begin{equation*}
0 \goto \Tor_2^{A}(Q/f, L(x)) \goto Q[f]\tensor_{A} L(x) \goto Q_x[f] \goto \Tor_1^A(Q/f, L(x)) \goto 0.
\end{equation*}
\item If $i\geq 1$ then $\Tor_{i+2}^A(Q/f,L(x)) \simeq \Tor_i^A(Q[f],L(x))$.
\end{enumerate}
\end{proposition}
\begin{proof}
Consider the complex $N_{\bullet}$ of $\cal R_A$-modules  given by $N_{\bullet} = [Q \overto{f} Q]$ and its base change $N_{\bullet,x} = [Q_x \overto{f} Q_x]$ to $x \in X$. Apply the K\"unneth spectral sequence again with $M = L(x)$. The homology groups are $H_0(N_{\bullet}) = Q/fQ$ and $H_1(N_{\bullet}) = Q[f]$ and vanish in degree $i\geq 2$ (the same for $N_{\bullet,x}$). From the spectral sequence we get for all $i\geq 2$ a short exact sequence
\begin{multline*}
0 \goto \coker\left(\Tor_{i+1}^A(Q/f,L(x)) \goto \Tor_{i-1}^A(Q[f],L(x))\right) \goto H_{i}(N_{\bullet,x})\\ \goto \ker\left(\Tor_{i}^A(Q/f,L(x)) \goto \Tor_{i-2}^A(Q[f],L(x))\right) \goto 0.
\end{multline*}
The middle term vanishes since $i\geq 2$ and thus by induction we see part (b) is true and that $\Tor_{2}^A(Q/f,L(x)) \inject Q[f]\tensor_A L(x)$, making the sequence in (a) exact on the left. The spectral sequence taken when $p+q = 1$ gives a short exact sequence
\begin{equation*}
0 \goto \coker\left(\Tor_2^A(Q/f,L(x)) \goto Q[f] \tensor_{A} L(x)\right) \goto Q_x[f] \goto \Tor_1^A(Q/f,L(x)) \goto 0,
\end{equation*}
which shows the rest of (a).
\end{proof}

Following these two general base change theorems we can begin to set up a cohomology and base change framework. Recall from the end of Section \ref{subsec:robba-ring} that if $Q$ is a generalized $(\varphi,\Gamma_K)$-module over $X$ and $U = \Sp(B) \ci X$ is an admissible open affinoid subdomain then $\restrict{Q}{U} := Q\tensor_{\cal R_X} \cal R_U$.
\begin{definition}
If $Q$ is a generalized $(\varphi,\Gamma_K)$-module over $A$ then we say $Q$ has finite cohomology if for every affinoid subdomain $U = \Sp(B) \ci X$, $H^i(\restrict{Q}{U})$ is a finite $B$-module for $i=0,1,2$.
\end{definition}

\begin{remark}
If $Q$ is a generalized $(\varphi,\Gamma_K)$-module over $A$ with finite cohomology, then $\restrict{Q}{U}$ is a generalized $(\varphi,\Gamma_K)$-module over $B$ with finite cohomology for all affinoid subdomains $U = \Sp(B) \ci X$.
\end{remark}

\begin{theorem}[Kedlaya-Pottharst-Xiao]\label{theorem:finiteness}
If $Q$ is a generalized $(\varphi,\Gamma_K)$-module over $A$ then $Q$ has finite cohomology in the following situations:
\begin{enumerate}
\item $Q$ is a $(\varphi,\Gamma_K)$-module;
\item $Q$ is of the form $\coker(Q_1 \overset{e}{\inject} Q_2)$ where both $Q_1$ and $Q_2$ have finite cohomology and $e$ is $(\varphi,\Gamma_K)$-equivariant.
\end{enumerate}
\end{theorem}
\begin{proof}
If $Q$ is a bona fide $(\varphi,\Gamma_K)$-module over $X$ then it is also a bona fide $(\varphi,\Gamma_K)$-module over $U$ and thus each cohomology group $H^i(\restrict{Q}{U})$ is a finite $B$-module by the main theorem of  \cite{KedlayaPottharstXiao-Finiteness}.

Now suppose that $Q_1$ and $Q_2$ have finite cohomology and that $e: Q_1 \goto Q_2$ is an injective $(\varphi,\Gamma_K)$-equivariant map. Since $U\ci X$ is an affinoid subdomain, Proposition \ref{prop:exactness-restrict} implies that we have a short exact sequence of generalized $(\varphi,\Gamma_K)$-module over $U$
\begin{equation}\label{eqn:assoc-sequence}
0 \goto \restrict{Q_1}{U} \goto \restrict{Q_2}{U} \goto  \restrict{Q}{U} \goto 0.
\end{equation}
The finiteness of each $H^i(\restrict{Q}{U})$ as a $B$-module now follows from the finiteness each $H^i(\restrict{Q_j}{U})$ ($i=0,1,2$ and $j=1,2$) and the long exact sequence in cohomology associated to the sequence \eqref{eqn:assoc-sequence}.
\end{proof}

If $Q$ is a generalized $(\varphi,\Gamma_K)$-module over $A$ then we define functions on $X$ by the formula
\begin{equation*}
d_Q^i(x) := \dim_{L(x)} H^i(Q_x).
\end{equation*}
Notice the fiber is taken prior to taking cohomology.

If $Q$ has finite cohomology then Nakayama's lemma, together with the fact that affinoid algebras are Jacobson \cite[Proposition 6.1.1/3]{BGR}, implies that $x \mapsto \dim_{L(x)} H^i(Q)\tensor_A L(x)$ is upper semi-continuous on $X$ and, since $X$ is reduced, locally constant if and only if $H^i(Q)$ is flat. For $i$ fixed, we will say that $H^i(Q)$ satisfies base change if the natural map $H^i(Q)\tensor_{A} L(x) \goto H^i(Q_x)$ is an isomorphism for all $x \in X$.

\begin{proposition}\label{prop:base-change-framework}
If $Q$ is a nearly flat generalized $(\varphi,\Gamma_K)$-module with finite cohomology and $x \mapsto d^i_Q(x)$ is locally constant for $i\geq k$ then $H^i(Q)$ is flat over $X$ for $i\geq k$ and satisfies base change for $i\geq k-1$.
\end{proposition}
\begin{proof}
One argues by descending induction on $k$. Since the cohomology vanishes in degrees $k\geq 3$, the proposition is vacuous for $k\geq 4$. When $k = 3$, $H^2(Q)$ is a cokernel, by definition, and thus always satisfies base change.

Fix $k\leq 2$ and assume the result is true for $k+1$. Since the hypotheses for $k$ imply those of $k+1$, the induction hypothesis implies that $H^k(Q)$ satisfies base change. By assumption, $d_Q^k(x)$ is locally constant. Since $X$ is reduced, Nakayama's lemma implies that $H^k(Q)$ is flat over $X$. The fact that $H^{k-1}(Q)$ satisfies base change now follows from Proposition \ref{prop:coh-kunneth}.
\end{proof}

Recall that if $X$ is a reduced rigid space and $x_0 \in X(\bar \Q_p)$ then a subset $Z \ci X(\bar \Q_p)$ is said to accumulate at $x_0$ if there exists a neighborhood basis of affinoid neighborhoods $U$ of $x_0$ such that $Z \intersect U$ is Zariski dense in $U$ for all $U$.

\begin{corollary}\label{cor:free-and-base-change}
Let $Q$ be a nearly flat generalized $(\varphi,\Gamma_K)$-module with finite cohomology. If $x \in X$ and there exists a Zariski dense subset $Z \ci X(\bar \Q_p)$ accumulating at $x$ such that, for each $0\leq i\leq 2$, $d^i_Q(x) = d^i_Q(u)$ for all $u \in Z$ then $H^i(\restrict{Q}{U})$ is flat and satisfies base change for $0\leq i \leq 2$ for all sufficiently small affinoid subdomains $x \in U \ci X$.
\end{corollary}
\begin{proof}
$H^2(Q)$ always satisfies base change, whence $u \mapsto d_Q^2(u)$ is upper semi-continuous on $X$. Since $Z$ is Zariski dense, and $d_Q^2(x) = d_Q^2(u)$ for $u \in Z$, we may shrink $X$ and assume that $d_Q^2(-)$ is constant on $X$. It follows from Proposition \ref{prop:base-change-framework} that $H^2(Q)$ is flat and $H^1(Q)$ satisfies base change. By assumption on $Z$ accumulating at $X$, the hypotheses of the proposition remain true after we've shrunk $X$. Thus we may re-do the same proof to show the result for $i=1$ and then $i = 0$.
\end{proof}

\section{Triangulated families}\label{sec:triangulated-families}
Here we introduce triangulated families: families of generalized $(\varphi,\Gamma_K)$-modules which point-by-point have a triangulation. They aren't the most natural families to consider, as the data are given pointwise, but they will be a useful intermediary for Theorem \ref{thm:main-theorem}. Throughout this section we write $X = \Sp(A)$ for a reduced rigid analytic affinoid space over $L$.

\subsection{Triangulated families}
\begin{definition}
A pointwise triangulated family of generalized $(\varphi,\Gamma_K)$-modules with torsion centered at $x_0 \in X$ is 
\begin{itemize}
\item A generalized $(\varphi,\Gamma_K)$-module $Q$ over $X$;
\item An ordered tuple $((s_{j,\tau})_{\tau})_{1\leq j\leq d}$ of integers (called the torsion weights);
\item An ordered tuple  $(\delta_{j})_{1\leq j\leq d}$  of continuous characters $\delta_j: K^\x \goto \Gamma(X,\cal O)^\x$ (called the parameter);
\item A Zariski dense set of points $X^{\nc} \ci X(\bar \Q_p)$ (called the non-critical points);
\item A point $x_0 \in X(\bar \Q_p)$ (called the center).
\end{itemize}
such that
\begin{enumerate}[(TF1)]
\item If $y \in X^{\nc} \union \set{x_0}$ then $\wt_\tau(\delta_{1,y}) < \dotsb < \wt_\tau(\delta_{d,y})$ are distinct integers and $\set{s_{j,\tau}}_{j} = \set{\wt_{\tau}(\delta_{j,x_0})}_j$ (as sets) for all $\tau \in \Sigma_K$.\label{axiom:tf-weight}
\item If $i < j$ and $x \in X(\bar \Q_p)$ then $\delta_{i,x}^{-1}\delta_{j,x}$ is generic up to homothety;\label{axiom:tf-generic}
\item For each $x \in X(\bar \Q_p)$, there exists on $Q_{x}$ a standard triangulation of torsion length $i(x_0)$ and free length $d-i(x_0)$, with $0 \leq i(x_0) \leq d$, independent of $x$, whose free parameter is term-by-term homothetic to $(\delta_{j,x})_{i(x_0)<j \leq d}$;\label{axiom:tf-triangulation}
\item The standard triangulation $Q_{x_0,\bullet}$ has invariants:\label{axiom:tf-trix}
\begin{itemize}
\item[-] torsion exponents $((s_{j,\tau}-\wt_\tau(\delta_{1,x_0}))_{\tau})_{1 \leq j\leq i(x_0)}$
\item[-] torsion weights $((s_{j,\tau})_\tau)_{1 \leq j\leq i(x_0)}$
\item[-] free parameter $\left(\delta_{j,x_0}\prod_{\tau} z_\tau^{\wt_\tau(\delta_{j,x_0})-s_{j,\tau}}\right)_{i(x_0) < j \leq d}$
\end{itemize}
\item If $y \in X^{\nc}$ then the triangulation $Q_{y,\bullet}$ has invariants\label{axiom:tf-triy}:
\begin{itemize}
\item[-] torsion exponents $((\wt_\tau(\delta_{j,x_0})-\wt_\tau(\delta_{1,x_0}))_{\tau})_{1 \leq j\leq i(x_0)}$
\item[-] torsion weights $((\wt_\tau(\delta_{j,y}))_{\tau})_{1 \leq j\leq i(x_0)}$
\item[-] free parameter $(\delta_{j,y})_{i(x_0) < j \leq d}$
\end{itemize}
\item For each $C > 0$ the set of points \label{axiom:tf-accumulates}
\begin{equation*}
X^{\nc}_{C} = \set{y \in X^{\nc} \st \wt_{\tau}(\delta_{j,y}) - \wt_\tau(\delta_{j-1,y}) > C \text{ for all $2\leq j\leq d$} }
\end{equation*}
accumulates at $x_0$.
\end{enumerate}
\end{definition}
\begin{remark}
The axiom (TF\ref{axiom:tf-triy}) does not have a typo. The torsion exponents are given in terms of weights of characters at the point $x_0$ and are independent of $y \in X^{\nc}$. This is why $x_0$ is called the ``center'' of the torsion.
\end{remark}
\begin{remark}
We stress that if $Q$ is a pointwise triangulated family of generalized $(\varphi,\Gamma_K)$-modules over $X$ then there is no reason to believe that $x \mapsto Q_{x,j}$ defines a generalized $(\varphi,\Gamma_K)$-module $Q_j$ over $X$; the datum of the standard triangulations is really given only point-by-point.
\end{remark}
We will dwell further beyond these remarks. First, $i(x_0)$ is the torsion length of the standard triangulation $Q_{x_0,\bullet}$ but beware that $Q_{x_0}$ may actually be torsion free, even if $i(x_0) > 0$. Indeed, our definitions allow for successive quotients in $Q_{x_0,\bullet}$ to be zero. In particular, if $i(x_0) = 1$ and $s_{1,\tau} = \wt_\tau(\delta_{1,x_0})$ for all $\tau \in \Sigma_K$ then $Q_{x_0}$ is torsion free. This situation doesn't arise in our applications, but we mention it because the remark applies equally well to $y \in X^{\nc}$. Indeed, if $y \in X^{\nc}$ then (TF\ref{axiom:tf-trix}) says that the torsion exponents of $Q_{y,1}$ are given by $\wt_\tau(\delta_{1,x_0}) - \wt_\tau(\delta_{1,x_0}) = 0$ for all $\tau \in \Sigma_K$ and thus $Q_{1,y} = (0)$.

More generally, when $i(x_0) > 0$ and $y \in X^{\nc}$ there will be $i(x_0) -1$ distinct torsion steps in the standard triangulation $Q_{y,\bullet}$. On the other hand, in practice, $Q_{x_0,\bullet}$ will have $i(x_0)$ torsion steps. However, consider the situation where for all $\tau \in \Sigma_K$ there exists a $j$ such that $j \leq i(x_0)$ and $s_{j,\tau} = \wt_\tau(\delta_{1,x_0})$. Then, for each $\tau \in \Sigma_K$ we have an equality of $\tau$-torsion lengths $\ell_\tau(Q_y) = \ell_\tau(Q_{x_0})$ at $x_0$ versus $y \in X^{\nc}$. Thus the discrepancy in the number of torsion steps is really an artifact of how we are doing the bookkeeping. We will see in the course of proving Theorem \ref{thm:main-theorem} that it can even happen that the $t_\tau$-torsion submodule of $Q_{x_0}$ is isomorphic, as an $\cal R_L$-module, to the $t_\tau$-torsion submodule of $Q_y$ for one (and thus all) $y \in X^{\nc}$. 

\begin{example}
Theorem \ref{thm:make-torsion} below will explain how to build new pointwise triangulated families out of old ones. But let us motivate the definition of pointwise triangulated families, and the need for Theorem \ref{thm:make-torsion}, with an example previewing the applications in Section \ref{sec:variation}.

Consider a bona fide $(\varphi,\Gamma_K)$-module $D$ of rank $d$ over $X$  and assume that it is actually a pointwise triangulated family $D$ centered at $x_0 \in X$ with $i(x_0) = 0$. For example, you could start with a densely pointwise strictly trianguline $(\varphi,\Gamma_K)$-module in the sense of \cite[Section 6.3]{KedlayaPottharstXiao-Finiteness} ($D$ is triangulated at every point following \cite[Theorem 6.3.13]{KedlayaPottharstXiao-Finiteness}). 

Starting from $D$ (with the given parameter $(\delta_1,\dotsc,\delta_d)$) one can construct locally on $X$, near $x_0$, a $(\varphi,\Gamma_K)$-equivariant morphism $\cal R_X(\delta_1) \overset{\bf e}{\inject} D$ using \cite[Theorem 6.3.9]{KedlayaPottharstXiao-Finiteness}. Moreover, we can assume that the base change $\bf e_x$ to any point $x$ is still injective. If we set $Q = \coker(\bf e)$ then $Q$ provides an example of a pointwise triangulated family of {\em generalized} $(\varphi,\Gamma_K)$-modules which is not necessarily a $(\varphi,\Gamma_K)$-module. It is even nearly flat.

At the points $y \in X^{\nc}$, the fiber $Q_y$ will be a bona fide $(\varphi,\Gamma_K)$-module of rank $d-1$. At $x_0$ however, this may not be the case. The $(\varphi,\Gamma_K)$-module $D_{x_0}$ is triangulated by a triangulation whose first step is $t^{-s}\cal R_{L(x_0)}(\delta_{1,x_0})$ for some integer $s\geq 0$. When $s > 0$, which is the case when the triangulation of $D_{x_0}$ is {\em critical}, we see that the fiber $Q_{x_0}$ has a non-trivial torsion submodule which is killed by $t^s$.   Nevertheless, the generalized $(\varphi,\Gamma_K)$-module $Q$ is still a pointwise triangulated family of $(\varphi,\Gamma_K)$-modules over $X$ with torsion centered at $x_0$, and $i(x_0) = 1$. One can then hope to iterate this process, in the style of \cite{KedlayaPottharstXiao-Finiteness}, using these more general families.
\end{example}

Let's continue now with constructing new pointwise triangulated families out of old ones in general.

\begin{theorem}\label{thm:make-torsion}
If $Q$ is a nearly flat pointwise triangulated family with torsion centered at $x_0 \in X(\bar \Q_p)$, free length $d-i(x_0) > 0$ and finite cohomology then there exists an  affinoid neighborhood $x_0 \in U \ci X$ and a short exact sequence
\begin{equation*}
0 \goto \prod_\tau t_\tau^{\wt_\tau(\delta_{i(x_0)+1,x_0})-\wt_\tau(\delta_{1,x_0})}\cal R_U(\delta_{i(x_0)+1}) \goto \restrict{Q}{U} \goto Q' \goto 0
\end{equation*}
of generalized $(\varphi,\Gamma_K)$-modules over $U$.
Moreover, $Q'$ is a nearly flat pointwise triangulated family with torsion centered at $x_0$, free length $d-i(x_0) - 1$ and finite cohomology whose given data is the same as $Q$. Thus $Q'$ as in Theorem \ref{thm:make-torsion} satisfies (TF\ref{axiom:tf-weight}) - (TF\ref{axiom:tf-accumulates}) with $i(x_0)$ replaced by $i(x_0) + 1$.
\end{theorem}

\begin{proof}
This is an application of our cohomology and base change framework. To shorten notation, let $k_{j,\tau} = \wt_\tau(\delta_{j,x_0})$ and $\delta = \delta_{i(x_0)+1}\prod_\tau z_\tau^{k_{i(x_0)+1,\tau}-k_{1,\tau}}$. Note that $\wt_\tau(\delta) = k_{1,\tau}$ for each $\tau \in \Sigma_K$. We're going to compute the cohomology $H^{\bullet}(Q(\delta^{-1}))$.

We begin by computing the cohomology at $x_0$. First, $H^2(Q_{x_0}(\delta_{x_0}^{-1})) = (0)$ by (TF\ref{axiom:tf-generic}) and the vanishing of cohomology in degree two for torsion $(\varphi,\Gamma_K)$-modules. By the Euler--Poincar\'e formula it remains to compute the cohomology in degree zero. Let $\twid \delta$ be the character
\begin{equation*}
\twid \delta = \delta_{i(x_0)+1}\prod_{\tau} z_\tau^{k_{i(x_0)+1,\tau}-s_{i(x_0)+1,\tau}} \cdot 
\end{equation*}
By (TF\ref{axiom:tf-trix}), $\twid \delta$ is the first character in the free parameter of the standard triangulation $Q_{x_0,\bullet}$ and we see easily that $\delta = \left(\prod_{\tau} z_\tau^{s_{i(x_0)+1,\tau} - k_{1,\tau}}\right) \twid \delta$. For each $j \leq i(x_0)+1$ set $r_{j,\tau} = s_{j,\tau} - k_{1,\tau}$. According to (TF\ref{axiom:tf-trix}), for $j\leq i(x_0)$, these are the torsion exponents of the standard triangulation $Q_{x_0,\bullet}$. Moreover, if $j \leq i(x_0)$ then
\begin{equation*}
s_{j,\tau} - \wt_\tau(\twid \delta) = -r_{i(x_0)+1,\tau} + r_{j,\tau}.
\end{equation*}
Thus Proposition \ref{lemm:stupid-lemma}, whose non-homothetic hypothesis is valid by (TF\ref{axiom:tf-generic}),  implies that there exists a unique, up to scalar, $(\varphi,\Gamma_K)$-equivariant inclusion 
\begin{equation*}
\cal R_{L(x_0)}(\delta_{x_0}) = \prod_\tau t_\tau^{r_{i+1,\tau}}\cal R_{L(x_0)}(\twid \delta) \inject Q_{x_0},
\end{equation*}
so $\dim_{L(x_0)}H^0(Q_{x_0}(\delta_{x_0}^{-1})) = 1$. Moreover, Proposition \ref{lemm:stupid-lemma} also explicitly describes the induced standard triangulation on the quotient $Q_{x_0}/\cal R_{L(x_0)}(\delta_{x_0})$.

We now compute the cohomology at (some) points $y \in X^{\nc}$. As above, $H^2(Q_y(\delta_y^{-1})) = (0)$ by (TF\ref{axiom:tf-generic}) and the vanishing of cohomology in degree two for torsion $(\varphi,\Gamma_K)$-modules. Reset the definitions from the previous paragraph and make $r_{j,\tau} := k_{j,\tau} - k_{1,\tau}$ for all $j\leq i(x_0) + 1$. Once again, if $j \leq i(x_0)$ then $(r_{j,\tau})_\tau$ gives the torsion exponents of the standard triangulation $Q_{y,\bullet}$. For all $j\leq i(x_0)$, we easily compute
\begin{equation*}
\wt_\tau(\delta_{j,y}) - \wt_\tau(\delta_y) = \wt_\tau(\delta_{j,y}) - \wt_\tau(\delta_{i(x_0)+1,y}) + r_{i(x_0)+1,\tau}.
\end{equation*}
Choose a $C$ so large that if $y \in X^{\nc}_C$ then for all $\tau$ and $1\leq j \leq i(x_0)$, the negative integer $\wt_\tau(\delta_{j,y}) - \wt_\tau(\delta_{i(x_0)+1,y})$ is not among the finitely many values $\set{-r_{i+1,\tau},\dotsc, r_{j,\tau} - r_{i+1,\tau}- 1}$. Then by (TF\ref{axiom:tf-generic}), we can apply Proposition \ref{lemm:stupid-lemma} and conclude that if $y \in X^{\nc}_{C}$ then $H^0(Q_y(\delta_y^{-1}))$ is one-dimensional over $L(y)$. Again, Proposition \ref{lemm:stupid-lemma} also explains how a non-zero morphism $\cal R_{L(y)}(\delta_y) \inject Q_y$ will induce a standard triangulation on $Q_y$ by replacing the subquotient $\cal R_{L(y)}(\delta_{i(x_0)+1,y})$ of $Q_{y}$ by 
\begin{equation*}
\cal R_{L(y)}(\delta_{i(x_0)+1,y})/\cal R_{L(y)}(\delta) = \cal R_{L(y)}(\delta_{i(x_0)+1,y})/(\prod_{\tau} t_{\tau}^{k_{i(x_0)+1,\tau}-k_{1,\tau}}).
\end{equation*} By (TF\ref{axiom:tf-accumulates}) we can replace $X^{\nc}$ by $X^{\nc}_{C}$ and assume that $H^0(Q_y(\delta_y^{-1}))$ has constant dimension over all of $X^{\nc}$.

We now go back to the entire family. By the Euler--Poincar\'e formula the function $y\mapsto \dim_{L(y)}H^1(Q_y(\delta_y^{-1}))$ is constant on $X^{\nc}$ and agrees with $\dim_{L(x_0)}H^1(Q_{x_0}(\delta_{x_0}^{-1}))$. Since $Q$ is nearly flat and has finite cohomology, Corollary \ref{cor:free-and-base-change} implies we can choose an open neighborhood $x_0 \in U \ci X$ so that each $H^i(\restrict{Q(\delta^{-1})}{U})$ is free over $U$ and satisfies base change for each $i$. We now replace $X$ by such a $U$.

Choose a basis vector $\bf e \in H^0(Q(\delta^{-1}))$. Since $H^0(Q(\delta^{-1}))$ satisfies base change, if we specialize $\bf e$ to either $x=x_0$ or $x=y \in X^{\nc}$, we get an injective morphism $\bf e_x: \cal R_{L(x)}(\delta_x) \inject Q_{x}$ by the previous two paragraphs. By Lemma \ref{lemma:injective}(a) below we conclude that $\bf e$ is also injective. Let $Q' = \coker(\bf e)$ so that there is a short exact sequence
\begin{equation}\label{eqn:theorem5ses}
0 \goto \cal R_X(\delta) \overto{\bf e} Q \goto Q' \goto 0.
\end{equation}
Since $Q$ has finite cohomology, by assumption, and $\cal R_X(\delta)$ has finite cohomology by Theorem \ref{theorem:finiteness}(a),  $Q'$ has finite cohomology by Theorem \ref{theorem:finiteness}(b).

We also have to show that $Q'$ is nearly flat (after possibly shrinking $X$ more). Since $Q$ is nearly flat over $X$ and $\cal R_X$ is flat over $X$, it suffices to show that we can shrink $X$ around $x_0$ so that $\bf e_x$ is injective for all $x$. Since $\bf e_{x_0}$ is injective it suffices by Lemma \ref{lemma:injective}(b) to show that $x \mapsto \rank_{\cal R_x[1/t]} Q_x[1/t]$ is constant on $X$. But this follows from (TF\ref{axiom:tf-triangulation}), which implies that $Q_{x}[1/t]$ is finite free over $\cal R_{L(x)}[1/t]$ of rank $d-(i(x_0)+1)$ independent of $x$.

Finally, we need to check that $Q'$ is a triangulated family with the same data as $Q$, except $i(x_0)$ replaced by $i(x_0)+1$. The three axioms (TF\ref{axiom:tf-weight}), (TF\ref{axiom:tf-generic}) and (TF\ref{axiom:tf-accumulates}) don't depend on $Q$, so those are still true. The axiom (TF\ref{axiom:tf-triangulation}) is true because of \eqref{eqn:theorem5ses} and the fact that $Q'$ is nearly flat, so that all the base changes $\bf e_x$ are injective.  Checking either (TF\ref{axiom:tf-trix}) or (TF\ref{axiom:tf-triy}) follows from our use of Proposition \ref{lemm:stupid-lemma} at the points $x = x_0$ and $x = y$.
\end{proof}

There were two points unresolved in the previous theorem, both of which we resolve with the following lemma. For the second part, and more in the following results, we will make use of the appendix on Nakayama's lemma.

\begin{lemma}\label{lemma:injective}
Let $X$ be a reduced affinoid space, $Q$ a generalized $(\varphi,\Gamma_K)$-module over $X$ and $f: \cal R_X \goto Q$ a $(\varphi,\Gamma_K)$-equivariant map. 
\begin{enumerate}
\item If $X' \ci X(\bar \Q_p)$ is Zariski dense in $X$ and the specialization $f_u: \cal R_{L(u)} \goto Q_u$ is injective for all $u \in X'$ then $f$ is injective. 
\item If $x \mapsto \rank_{\cal R_{x}[1/t]} Q_x[1/t]$ is constant on $X$, $x_0 \in X$ and $f_{x_0}$ is injective then there exists an affinoid neighborhood ${x_0} \in U \ci X$ such that $f_u$ is injective for all $u \in U$.
\end{enumerate}
\end{lemma}
\begin{remark}
In part (b), $Q_x[1/t]$ is automatically free over $\cal R_x[1/t]$ for each $x$ by Lemma \ref{lemma:torsion-free}.
\end{remark}
\begin{proof}
We need to make use of the recollection given in Section \ref{subsec:robba-ring}. Choose an $r_0$ so that $Q$ arises via base change from $X^{r_0}$ and $f$ arises from a $(\varphi,\Gamma_K)$-equivariant map $f^{r_0}: \cal R_X^{r_0} \goto Q^{r_0}$ as well. 

We first prove (a). It suffices to show that $f^{r_0}$ is injective. If $0 < s < r_0$ then let $f^{[s,r_0]}$ denote the induced map
\begin{equation*}
\cal R_X^{[s,r_0]} \overto{f^{[s,r_0]}} Q^{[s,r_0]} := Q^{r_0} \tensor_{\cal R_X^{r_0}} \cal R_X^{[s,r_0]}.
\end{equation*}
By \cite[Lemma 2.1.4(2)]{KedlayaPottharstXiao-Finiteness} it suffices to show that $f^{[s,r_0]}$ is injective. 

We've now reduced to working over a closed relative annulus $X^{[s,r_0]}$ whose rigid functions $\cal R_X^{[s,r_0]}$ are, in particular, noetherian. Write $I$ for the image of $f^{[s,r_0]}$ in $Q^{[s,r_0]}$. If $u \in X'$ then $\cal R_{L(u)}^{[s,r]} \goto I\tensor_A  L(u)$ is injective, as it factors $f_u$ and we've assumed that $f_u$ is injective. It is also surjective since tensor product is right exact. Thus it is an isomorphism. 

We deduce from Nakayama's lemma, applied to the finite module $I$ over the noetherian ring $\cal R_X^{[s,r_0]}$ that $\dim_{L(v)} I_v \geq 1$ for all $v \in X^{[s,r_0]}$ (note: $v$ is in the relative annulus, not just $X$, and $I_v := I\tensor_{\cal R_X^{[s,r_0]}} L(v)$). But $\dim_{L(v)} I_v \leq 1$ for all $v \in X^{[s,r_0]}$ since $I_v$ is a quotient of something free of rank one over $X^{[s,r_0]}$. Since $X$ is reduced, so is $X^{[s,r_0]}$ and we just showed that $v \mapsto \dim_{L(v)} I_v$ is constant on the relative annulus $X^{[s,r_0]}$. Thus $I$ must be flat over $X^{[s,r_0]}$ by Nakayama's lemma again. In particular, $\cal R_{X}^{[s,r_0]} \goto I$ is an isomorphism, as was to be shown.

Let's now prove part (b). We note that the proof of part (a) didn't use that $f$ was $(\varphi,\Gamma_K)$-equivariant. Consider $C^{r_0} = \coker(\cal R_X^{r_0} \overto{f^{r_0}} Q^{r_0})$. By Corollary \ref{cor:cat-abelian}, $C^{r_0}$ is a generalized $(\varphi,\Gamma_K)$-module. If $u \in X$ then $C_u^{r_0}[1/t]$ is free over $\cal R_{L(u)}^{r_0}[1/t]$ by Lemma \ref{lemma:torsion-free}, so it makes sense to consider the function $\rank(u):= \rank_{\cal R_{L(u)}^{r_0}[1/t]}(C_u^{r_0}[1/t])$. Moreover, if we choose any $0 < s < r_0$ then we also have\footnote{It is important here that $C^{r_0}$ is finitely presented, so that $C_u^{[s,r_0]} = C_u^{r_0} \tensor_{\cal R_{L(u)}^{r_0}} \cal R_{L(u)}^{[s,r_0]}$ uses the usual tensor product.}
\begin{equation}\label{eqn:reduce-to-affinoid}
\rank(u) = \rank_{\cal R_{L(u)}^{[s,r_0]}[1/t]} C_u^{[s,r_0]}[1/t].
\end{equation}
Consider such a choice of $s$ made now. 

Since $\rank_{\cal R_{L(u)}^{r_0}[1/t]} Q_u^{r_0}[1/t] =: q$ is constant on $X$, we know that $\rank(u)$ is either $q$ or $q-1$. Since $f_{x_0}^{r_0}$ is injective, so is $f_{x_0}^{r_0}[1/t]$ and thus $\rank(x_0) = q-1$ is the minimal possible value. Since $\cal R_X^{[s,r_0]}$ is affinoid, the expression \eqref{eqn:reduce-to-affinoid} and Proposition \ref{prop:nakayama-lemma}(a) together imply that we may replace $X$ by an affinoid subdomain containing $x_0$ so that $\rank(u) = q-1$ for all $u \in X$. But that clearly is equivalent to $\ker f^{r_0}_u[1/t] = (0)$ for all $u \in X$. Since $f_u^{r_0}: \cal R_{L(u)}^{r_0} \goto Q_u^{r_0}$ has source $\cal R_{L(u)}$, Proposition \ref{prop:rank-one-subs} and the fact that $f_u^{r_0}[1/t]$ is injective for all $u\in X$ implies that $\ker f_u^{r_0} = (0)$ for all $u\in X$ (and thus $f_u$ is injective also).
\end{proof}

\subsection{Killing torsion}
Notice that Theorem \ref{thm:make-torsion} possibly introduces torsion $(\varphi,\Gamma_K)$-modules into the picture. Thus its utility rests on being able to kill torsion in certain pointwise triangulated families. This is achieved in Corollary \ref{cor:tau-balanced}. First we need preparation. The following is an application of Nakayama's lemma in the appendix.

\begin{lemma}\label{lemma:ranks-by-t}
Let $\tau \in \Sigma_K$. Suppose that $Q$ is a generalized $(\varphi,\Gamma_K)$-module such that for each $x \in X$, $Q_x/t_\tau$ is finite free over $\cal R_{L(x)}/t_\tau$. Suppose that $x_0 \in X(\bar \Q_p)$ and $Z \ci X(\bar \Q_p)$ is a set of points accumulating at $x_0$ and $\rank_{\cal R_{x_0}/t_\tau} Q_{x_0}/t_\tau = \rank_{\cal R_z/t_\tau} Q_z/t_\tau$ for each $z \in Z$. Then there exists an affinoid subdomain $x_0 \in U \ci X$ such that $u\mapsto \rank_{\cal R_{L(u)}/t_\tau} Q_u/t_\tau$ is constant on $U$.
\end{lemma}
\begin{proof}
First, if necessary, replace $X$ by an affinoid subdomain $x_0 \in U\ci X$ so that $Z$ is Zariski dense in $U$. Then, we find an $r_0$ so that $Q$ arises from a $(\varphi,\Gamma_K)$-module $Q^{r_0}$ over $X^{r_0}$. Since $Q_u/t_\tau$ is finite free over $\cal R_{L(u)}/t_\tau$ for each $u \in X$, the same is true for $Q_u^{r_0}/t_\tau$ over $\cal R^{r_0}_{L(u)}/t_\tau$. (Since the Frobenius $\cal R_{L(u)}^{r_0} \goto \cal R_{L(u)}^{r_0/p}$ is faithfully flat and $Q^{r_0}$ is a generalized $(\varphi,\Gamma_K)$-module, it is enough to check $Q^r_u/t_{\tau}$ is finite free over $\cal R^r_{L(u)}/t_\tau$ for some $0 < r \leq r_0$ (possibly depending on $u$); this follows from knowing $Q_u/t_\tau$ is finite free over $\cal R_{L(u)}/t_\tau$ and $Q^r_u/t$ is finitely presented over $\cal R_{L(u)}^r/t$ for each $r$; compare with  \cite[Lemma 2.1.16]{KedlayaPottharstXiao-Finiteness}.)

Fix any $0 < s < r_0$. Then 
\begin{equation}\label{eqn:ranks-equal}
\rank_{\cal R_{L(u)}^{r_0}/t_\tau} Q^{r_0}_u/t_\tau = \rank_{\cal R_{L(u)}^{[s,r_0]}/t_\tau} Q^{[s,r_0]}_u/t_\tau
\end{equation}
for each $u \in X$. The ring $\cal R^{[s,r_0]}/t$ is reduced since $t$ is well-known to be a uniformizer at the at the points of the form $\zeta-1 \in \Af^1[s,r_0]$ where $\zeta$ is a $p$-power root of unity. Thus so is the factor ring $\cal R^{[s,r_0]}/t_\tau$. In particular,  $\cal R^{[s,r_0]}_X/t_\tau = \cal R^{[s,r_0]}/t_\tau \hat\otimes_{\Q_p} A$ is the completed tensor product of reduced affinoid $\Q_p$-algebras. It follows from Proposition \ref{prop:nakayama-lemma}(b) that the right hand side of \eqref{eqn:ranks-equal} has a minimum achieved on the Zariski dense subset $Z\ci X$. Since that minimum is achieved also at $x_0$, by assumption, Proposition \ref{prop:nakayama-lemma}(a) allows us to find an affinoid subdomain $x_0 \in U \ci X$ on which $u \mapsto \rank_{\cal R_{L(u)}^{r_0}/t_\tau} Q^{r_0}_u/t_\tau$ is constant.
\end{proof}

\begin{lemma}\label{lemma:killing-torsion}
Let $\tau \in \Sigma_K$. Suppose that $Q$ is a nearly flat generalized $(\varphi,\Gamma_K)$-module, $x_0 \in X(\bar \Q_p)$ and $Z \ci X(\bar \Q_p)$ is a set of points accumulating at $x_0$ such that:
\begin{enumerate}
\item there exists non-negative integers $s,r$ such that for all $u \in X(\bar \Q_p)$, $Q_u$ is triangulated by a standard triangulation with non-zero $\tau$-torsion exponents $(m_{i,\tau}(u))_{1\leq i \leq s}$ and free parameter of length $r$,
\item $z\mapsto \min_{i=1}^s m_{i,\tau}(z)$ is a constant $m$ on $Z$, and
\item $m \leq \min_{i=1}^s m_{i,\tau}(x_0)$ as well.
\end{enumerate}
Then there exists an open affinoid $x_0 \in U \ci X$ such that $\restrict{\left(Q/t_\tau^{m}\right)}{U}$ is flat over $\cal R_U/t_\tau^m$.
\end{lemma}
\begin{proof}
Let's begin by elucidating the first assumption. Since $t_\tau$ and $t_\sigma$ are maximally coprime if $\tau \neq \sigma$, the first assumption implies that
\begin{equation*}
Q_u \simeq \left(\bigdsum_{i=1}^s \cal R_{L(u)}/t_\tau^{m_{i,\tau}(u)}\right) \dsum \cal R_{L(u)}^{\dsum r} \dsum S_u
\end{equation*}
as an $\cal R_{L(u)}$-module, where $S_u[t_\tau] = S_u/t_\tau = (0)$ and each $m_{i,\tau}(u)$ is non-zero. 
\begin{claim}
For each $1 \leq j \leq m$ we may replace $X$ by an affinoid neighborhood $x_0 \in U \ci X$ such that $Q_u/t_\tau^{j}$ is free of rank $r+s$ over $\cal R_{L(u)}/t_\tau^j$.
\end{claim}
The proof will be given by induction in the next paragraph. Assuming the claim for the moment, let us finish the lemma. Fix $1 \leq j \leq m$. By the claim we may assume that $Q_u/t_{\tau}^{j}$ is free over $\cal R_{L(u)}/t_\tau^{j}$ with rank independent of $u$. We want to show that this implies $Q/t_\tau^j$ is flat over $\cal R_X/t_\tau^j$. This will follow from Lemma \ref{lemma:flat} and \cite[Lemma 2.1.8(2)]{KedlayaPottharstXiao-Finiteness}. Indeed, we may first spread out $Q$ to a finitely presented module $Q^{r_0}$ on a half-open annulus $X^{r_0}$. Having done that, Lemma \ref{lemma:flat} shows that it suffices to check that $Q^{[s,r_0]}/t_\tau^j$ is finite flat over $\cal R_X^{[s,r_0]}/t_\tau^j$ for each $0 < s < r_0$. To check that, we observe that $\cal R_X^{[s,r_0]}/t_\tau^j$ is the completed tensor product of an affinoid algebra $\cal R^{[s,r_0]}/t_\tau^j$ with a reduced affinoid algebra $A$ (where $X = \Sp(A)$). In particular, \cite[Lemma 2.1.8(2)]{KedlayaPottharstXiao-Finiteness} (which only requires {\em one} of the tensor-ands to be reduced) and the constancy $u\mapsto \rank_{\cal R_{L(u)}^{[s,r_0]}/t_\tau^j} Q_u^{[s,r_0]}/t_\tau^j$ (our assumption in the claim) implies that $Q^{[s,r_0]}/t_\tau^j$ is finite flat over $\cal R_X^{[s,r_0]}/t_\tau^j$.

\begin{proof}[Proof of claim] We now prove the claim by induction on $1\leq j\leq m$ starting with $j = 1$. Since $m_{i,\tau}(u) \geq 1$ for all $i$ and $u$, we see visibly that $Q_u/t_\tau$ is free over $\cal R_{L(u)}/t_\tau$ of rank $r + s$ for each $u \in X$. Let $1 \leq j < m$ and assume $Q_u/t_\tau^{j}$ is free over $\cal R_{L(u)}/t_\tau^{j}$ for all $u\in X$. If we show that we may shrink $X$ so that $j < m_{i,\tau}(u)$ for each $i$ and $u \in X$ then the same freeness will be true for $j+1$ and we'll be done.

First, $Q/t_\tau^{j}$ is flat over $\cal R_X/t_\tau^{j}$ by induction and the proof in the paragraph following the claim. Since $\cal R_X/t_\tau^{j}$ is flat over $X$, see \cite[Corollary 2.1.5]{KedlayaPottharstXiao-Finiteness}, we get that $Q/t_\tau^{j}$ is flat over $X$. Second, $Q$ is nearly flat over $X$ and thus Proposition \ref{prop:torsion-kunneth} implies that the $t_\tau^{j}$-torsion satisfies base change: for each $u \in X(\bar \Q_p)$,
\begin{equation*}
Q[t_\tau^{j}]_u = Q_u[t_\tau^{j}] \simeq \bigdsum_{i=1}^s  t_\tau^{\max(0,m_{i,\tau}(u)-j)}\cal R_{L(u)}/t^{m_{i,\tau}(u)}.
\end{equation*}
If we define $Q' := Q/Q[t_\tau^{j}]$ then $Q'$ is a generalized $(\varphi,\Gamma_K)$-module over $X$ and we see
\begin{equation}\label{eqn:Qprime_fiber}
Q'_u \simeq \left(\bigdsum_{i=1}^s \cal R_{L(u)}/t_\tau^{\max(0,m_i(u)-j)}\right) \dsum \cal R_{L(u)}^{\dsum r} \dsum S_u
\end{equation}
as a $\cal R_{L(u)}$-module. Note immediately that $Q'_u/t_\tau$ is free over $\cal R_{L(u)}/t_\tau$ for any $u$. Now specialize to $u = x_0$ or $u \in Z$. From the assumptions (b) and (c) we have that $Q'_u/t_\tau$ has rank $r+s$ at $u=x_0$ and at $u\in Z$. By Lemma \ref{lemma:ranks-by-t}, we may replace $X$ by an affinoid subdomain so that $Q'_u/t_\tau$ has precisely rank $r+s$ everywhere on $X$. The formula \eqref{eqn:Qprime_fiber} for the fiber $Q'_u$ easily implies that $m_{i}(u) - j > 0$, which is what we wanted to show.
\end{proof}
With the inductive step complete, the proof is finished.
\end{proof}

Recall we defined the $\tau$-torsion length $\ell_\tau(Q)$ of generalized $(\varphi,\Gamma_K)$-modules in Section \ref{subsect:generalized-triangulations}. If we know that $Q$ has a standard triangulation $Q_{\bullet}$ then its $\tau$-torsion length is the number terms in the associated graded with non-zero $\tau$-torsion exponent.

\begin{proposition}\label{prop:killing-tors}
Suppose that $Q$ is a nearly flat pointwise triangulated family with torsion centered at $x_0 \in X(\bar \Q_p)$ and non-critical points $X^{\nc}$. If $\tau \in \Sigma_K$ such that $\ell_\tau(Q_{x_0}) = \ell_\tau(Q_{y})$ at one (and hence all) $y \in X^{\nc}$ then there exists an affinoid neighborhood $x_0 \in U \ci X$ such that $Q[t_\tau^{\wt_\tau(\delta_{2,x_0})-\wt_\tau(\delta_{1,x_0})}]_x \simeq Q_x[t_\tau^{\wt_\tau(\delta_{2,x_0})-\wt_\tau(\delta_{1,x_0})}]$ for all $x \in U(\bar \Q_p)$ and $\restrict{Q}{U}[t_\tau^{\wt_\tau(\delta_{2,x_0})-\wt_\tau(\delta_{1,x_0})}]$ is a nearly flat generalized $(\varphi,\Gamma_K)$-module over $U$.
\end{proposition}
\begin{proof}
We're going to apply Lemma \ref{lemma:killing-torsion}. Note that $Q_x/t_\tau$ is free over $\cal R_{L(x)}/t_\tau$ for all $x \in X(\bar \Q_p)$. Moreover, the rank at $x = x_0$ agrees with the rank over the Zariski dense set $X^{\nc}$. By (TF\ref{axiom:tf-triangulation}), the free parameters of the family $Q$ have constant length $d-i(x_0)$ on all of $X$. Since $\rank_{\cal R_{L(x)}/t_\tau} Q_x/t_\tau = \ell_\tau(Q_x) + d-i(x_0)$, Lemma \ref{lemma:ranks-by-t} implies that we may shrink $X$ and assume that the $\tau$-torsion lengths are also constant on the entire family. In particular, the first hypothesis of Lemma \ref{lemma:killing-torsion} is verified. By (TF\ref{axiom:tf-triy}), the minimal torsion exponent at $y \in X_{\nc}$ is given by $\wt_\tau(\delta_{2,x_0})-\wt_\tau(\delta_{1,x_0})$, independent of $y$. Thus the second hypothesis of Lemma \ref{lemma:killing-torsion} is verified. On the other hand, by (TF\ref{axiom:tf-weight}), $\wt_\tau(\delta_{2,x_0})-\wt_\tau(\delta_{1,x_0})$ is also the smallest possible value for elements in the set $\set{s_{j,\tau}-\wt_\tau(\delta_{1,x_0}) \st 2 \leq j \leq d}$, which are the possible non-zero $\tau$-torsion exponents of the standard triangulation on $Q_{x_0}$. Thus the third hypothesis of Lemma \ref{lemma:killing-torsion} is satisfied.  We conclude by Lemma \ref{lemma:killing-torsion} that $\restrict{Q}{U}/t_\tau^{\wt_\tau(\delta_{2,x_0})-\wt_\tau(\delta_{1,x_0})}$ is flat over $\cal R_U/t_\tau^{\wt_\tau(\delta_{2,x_0})-\wt_\tau(\delta_{1,x_0})}$ for some affinoid neighborhood $U$ of $x_0$. Replacing $X$ by $U$ we're done by Proposition \ref{prop:torsion-kunneth}.
\end{proof}

\begin{corollary}\label{cor:tau-balanced}
Suppose $Q$ is a nearly flat pointwise triangulated family with torsion centered at $x_0 \in X(\bar \Q_p)$, non-critical points $X^{\nc}$ and for each $\tau \in \Sigma_K$, we have $\ell_\tau(Q_{x_0}) = \ell_\tau(Q_y)$ for one (and hence all) points $y \in X^{\nc}$. Then 
\begin{enumerate}
\item For each $\tau$ there exists a unique $1 \leq j_\tau \leq i(x_0)$ such that $s_{j_\tau,\tau} = k_{1,\tau}$.
\item There exists an affinoid neighborhood $x_0 \in U \ci X$ such that the quotient $Q':=\restrict{Q}{U}/\restrict{Q}{U}[\prod_\tau t_\tau^{\wt_\tau(\delta_{2,x_0})-\wt_\tau(\delta_{1,x_0})}]$ is a nearly flat pointwise and triangulated family of $(\varphi,\Gamma_K)$-modules over $U$ with torsion centered at $x_0$ and modified data
\begin{itemize}
\item torsion weights $(s_{j,\tau})_{j \neq j_\tau}$,
\item parameter $(\delta_2,\dotsc,\delta_{d})$,
\item $X^{\nc}$ replaced by $U \intersect X^{\nc}$.
\end{itemize}
In particular, the free length of $Q'$ is the same as the free length of $Q$.
\end{enumerate}
\end{corollary}
\begin{proof}
The fact that the $j_\tau$ exists is clear since the $\tau$-torsion lengths are constant on $X^{\nc}$ and concur with the lengths at $x_0$. Choose, by Proposition \ref{prop:killing-tors}, an affinoid neighborhood $x_0 \in U \ci X$ over which each torsion $\restrict{Q}{U}[t_\tau^{\wt_\tau(\delta_{2,x_0})-\wt_\tau(\delta_{1,x_0})}]$ satisfies base change and is a nearly flat generalized $(\varphi,\Gamma_K)$-module over $U$. Since the $t_\tau$ are maximally coprime within $\cal R$, the torsion $\restrict{Q}{U}[\prod_\tau t_\tau^{\wt_\tau(\delta_{2,x_0})-\wt_\tau(\delta_{1,x_0})}]$ is also a nearly flat generalized $(\varphi,\Gamma_K)$-module and satisfies base change over $U$.  Let $Q'$ be as in the statement of part (b). Then we see that $Q'$ is nearly flat over $U$, proving the first half of (b). If $x \in U$ then the fiber of $Q'$ at $x$ is computed as
\begin{equation*}
Q_x' = \coker(Q_x[\prod_\tau t_\tau^{\wt_\tau(\delta_{2,x_0})-\wt_\tau(\delta_{1,x_0})}] \inject Q_x).
\end{equation*}
The image, by definition, lands in the torsion part of $Q_x$ and hence can only effect the torsion exponents of the standard triangulation on $Q_x'$. Thus the rest of part (b) follows easily follows from the observation that
\begin{equation*}
\coker\left(t^{s-k_{2,\tau}}\cal R_L/t^{s-k_{1,\tau}}\cal R_L \goto \cal R_L/t^{s-k_{1,\tau}}\right)
\end{equation*}
has torsion exponent $s-k_{2,\tau}$ for any $s \geq k_{2,\tau} > k_{1,\tau}$.
\end{proof}

\section{$p$-adic variation in refined families}\label{sec:variation}

We are now ready to state and prove the main theorem of this article. The notion of refined families given below is inspired by \cite[Chapter 4]{BellaicheChenevier-Book} (see also \cite[Section 5]{Liu-Triangulations}). They arise naturally as arithmetic families of $(\varphi,\Gamma_K)$-modules over rigid spaces, for example eigenvarieties.

\subsection{Refined families}\label{subsec:refined-families}

Let $X = \Sp(A)$ be a reduced $L$-affinoid space.

\begin{definition}
A refined family of $(\varphi,\Gamma_K)$-modules of rank $d$ is
\begin{itemize}
\item[-] a $(\varphi,\Gamma_K)$-module $D$ of rank $d$ over $X$,
\item[-] an ordered tuple $(\delta_1,\dotsc,\delta_d):K^\x \goto A^\x$ of continuous characters, and
\item[-] a Zariski dense subset $X_{\cl} \ci X(\bar \Q_p)$ 
\end{itemize}
such that the following axioms hold:
\begin{enumerate}[{(RF1)}]
\item For each $x \in X(\bar \Q_p)$ and $\tau \in \Sigma_K$,
\begin{equation*}
\HT_\tau(D_x) = \set{\wt_\tau(\delta_{1,x}),\dotsc,\wt_\tau(\delta_{d,x})}.
\end{equation*}
We now label the Hodge--Tate--Sen weights by $\kappa_{i,\tau}(x) := \wt_\tau(\delta_{i,x})$. \label{axiom:RFwts}
\item \label{axiom:characters-crystalline} For each $x \in X_{\cl}$ and $i=1,\dotsc,d$, the character  $\delta_{i,x}$ is crystalline.
\item \label{axiom:crystalline} If $x \in X_{\cl}$ then $D_x$ is crystalline, the Hodge--Tate weights satisfy 
\begin{equation*}
\kappa_{1,\tau}(x) < \dotsb < \kappa_{d,\tau}(x),
\end{equation*}
for each $\tau \in \Sigma_K$, and the $\varphi^{f_K}$ eigenvalues $\set{\phi_1(x),\dotsc,\phi_d(x)}$ all live in $L(x)^\x$, are distinct, and given by
\begin{equation*}
\phi_i(x) = \delta_{i,x}(\varpi_K)\prod_{\tau \in \Sigma_K} \tau(\varpi_K)^{\kappa_{i,\tau}(x)}
\end{equation*}
for some (any) uniformizer $\varpi_K \in K^\x$.
\item By (RF\ref{axiom:crystalline}) and Proposition \ref{prop:refinements-triangulations}, every point $x\in X_{\cl}$ has a unique triangulation $P_{x,\bullet}$ corresponding to the ordering $(\phi_1(x),\dotsc,\phi_n(x))$ of (distinct) crystalline eigenvalues. Let
\begin{equation*}
X_{\cl}^{\nc} := \set{x \in X_{\cl} \st P_{x,\bullet} \text{ is a non-critical triangulation}}.
\end{equation*}
Then, for all $C > 0$ the set
\begin{equation*}
X_{\cl,C}^{\nc} := \set{x \in X_{\cl}^{\nc} \st C < \sum_{\tau \in \Sigma_K} \kappa_{i+1,\tau}(x) - \kappa_{i,\tau}(x)  \text{ for $i=1,\dotsc,n-1$}}
\end{equation*}
is Zariski dense in $X$ and accumulates at every point in $X_{\cl}$.\label{axiom:accumulate} \label{axiom:classicality}
\end{enumerate}
\end{definition}
We will often abuse language and call $D$ the refined family, with the ordered parameter $(\delta_i)_{i=1,\dotsc,d}$ and the subset $X_{\cl}$ understood. The subscript ``$\cl$'' is meant to mean classical, but note that there are slight restrictions. Indeed, a refined family for us has a distinctness hypothesis on not only the Hodge--Tate weights over $X_{\cl}$ but also the crystalline eigenvalues. In applications we will work with a slightly stronger condition.

\begin{definition}\label{defi:very-phi-regular}
Suppose that $D$ is a refined family of $(\varphi,\Gamma_K)$-modules over $X$. We say that $x \in X_{\cl}$ is very $\varphi$-regular if
\begin{enumerate}
\item $\phi_i(x) \neq p^{f_K} \phi_j(x)$ for each $1\leq i < j \leq d$, and
\item $\phi_1(x)\dotsb \phi_i(x)$ is a simple eigenvalue of $\varphi^{f_K}$ acting on $D_{\cris}(\wedge^i D_x)$ for each $1 \leq i \leq d$.
\end{enumerate}
The set of all very $\varphi$-regular points is denoted by $X_{\cl}^{\varphi_{\reg}}$.
\end{definition}

Just as axiom (RF\ref{axiom:accumulate}) says that classical points are well-approximated by points whose Hodge--Tate weights are extremely regular (i.e. far apart), the following proposition shows that we can, moreover, make such approximations by very $\varphi$-regular points as well. We let $X_{\cl,C}^{\nc,\varphi_{\reg}} = X_{\cl,C}^{\nc} \intersect X_{\cl}^{\varphi_{\reg}}$. 
\begin{proposition}
If $C > 0$ then $X_{\cl,C}^{\nc,\varphi_{\reg}}$ accumulates each every point in $X_{\cl}$.
\end{proposition}
\begin{proof}
Let $x \in X_{\cl}$. Since each character $\delta_i:K^\x \goto A^\x$ is continuous we may choose a neighborhood $U$ of $x$ so that the slopes $v_p(\delta_{i,u}(\varpi_K)) =: \nu_i$ are constant on $U$ for some (and hence any) uniformizer $\varpi_K$. Let 
\begin{equation*}
C' = \max\set{C, 1 + e_K(\nu_i-\nu_j) \st 1 \leq i,j \leq d}.
\end{equation*}
Since $C \leq C'$, $X_{\cl,C}^{\nc,\varphi_{\reg}} \cg X_{\cl,C'}^{\nc,\varphi_{\reg}}$. Thus it suffices to show, by (RF\ref{axiom:classicality}), that $X_{\cl,C'}^{\nc} \intersect U \ci X_{\cl,C'}^{\nc,\varphi_{\reg}}$. Let $u \in X_{\cl,C'}^{\nc} \intersect U$ and we will show $u$ is very $\varphi$-regular.
\begin{itemize}
\item Suppose $i < j$ and $\phi_i(u) = p^{f_K}\phi_j(u)$. Since $\phi_i(u) = \delta_{i,u}(\varpi_K)\prod_{\tau}\tau(\varpi_K)^{\kappa_{i,\tau}(u)}$ we can take $p$-adic valuations and get, since $u \in U$, that
\begin{equation*}
\sum_{\tau} \kappa_{j,\tau}(u) - \kappa_{i,\tau}(u) = e_K(\nu_i - \nu_j - f_K) < C'.
\end{equation*}
Thus $u \nin X_{\cl,C'}^{\nc}$, a contradiction.
\item If $\phi_1(u)\dotsb \phi_i(u)$ is not a simple eigenvalue on $D_{\cris}(\wedge^i D_u)$ then one may construct a list of pairs of integers $i_1 < j_1, \dotsc, i_s < j_s$ such that $\phi_{i_1}(u)\dotsb \phi_{i_s}(u) = \phi_{j_1}(u)\dotsb \phi_{j_s}(u)$ for some $s \geq 1$. Once again, taking slopes we get
\begin{equation*}
sC' < \sum_{b=1}^s \sum_{\tau \in \Sigma_K} \kappa_{j_b,\tau}(u) - \kappa_{i_b,\tau}(u) =  \sum_{b=1}^s \left(e_K(\nu_{i_b} - \nu_{j_b})\right) \leq sC',
\end{equation*}
a contradiction.
\end{itemize}
This concludes the proof.
\end{proof}

In the remainder of this subsection we show that point-by-point, a refined family of $(\varphi,\Gamma_K)$-modules may be triangulated (in an essentially unique way depending on the ordered parameter $(\delta_1,\dotsc,\delta_n)$). In particular, we show that a refined family is naturally a pointwise triangulated family (with no torsion). We begin by dealing with the axiom (TF\ref{axiom:tf-generic}).

\begin{lemma}\label{lemma:generic-cohomology}
If $x \in X_{\cl}^{\varphi_{\reg}}$ then there exists an open affinoid neighborhood $x \in U \ci X$ such that $\delta_{i,u}^{-1}\delta_{j,u}$ is generic up to homothety for all $u \in U(\bar \Q_p)$ and $i < j$.
\end{lemma}
\begin{proof}
Let $(r_\tau)_{\tau}$ be any tuple of integers and set $\eta = \delta_{i}^{-1}\delta_{j}\prod_{\tau} z_\tau^{-r_\tau}$. We have to show $\eta_u$ is generic for $u$ near $x$. Choose a uniformizer $\varpi_K$ of $K^\x$. Then by (RF\ref{axiom:crystalline}),
\begin{equation*}
\eta_x = \unr_{\varpi_K}(\phi_j(x)\phi_i(x)^{-1})\prod_{\tau \in \Sigma_K} z_\tau^{\kappa_{i,\tau}(x)-\kappa_{j,\tau}(x)-r_\tau}.
\end{equation*}
If $\eta_x$ is not generic then a comparison of weights shows that that
\begin{equation*}
\unr_{\varpi_K}(\phi_j(x)\phi_i(x)^{-1}) \in \set{1, \abs{N_{K/\Q_p}(\varpi_K)}}.
\end{equation*}
However, since $x$ is very $\varphi$-regular and $i < j$ this is explicitly ruled out. Thus $\eta_x$ is generic.

To conclude over an affinoid neighborhood we make use of cohomology and base change arguments. Consider the functions $d_\eta^i(u) := \dim_{L(u)} H^i(\eta_u)$. Since $H^2(\eta)$ satisfies base change, $u \mapsto d^2_\eta(u)$ is upper semi-continuous on $X$. Since $d_\eta^2(u)$ vanishes at $u = x$, as we showed in the previous paragraph, we may shrink $X$ and assume that $H^2(\eta) = (0)$. By Proposition \ref{prop:base-change-framework}, $H^1(\eta)$ satisfies base change and $u\mapsto d^1_\eta(u)$ is upper semi-continuous. But $d^0_\eta(u)$ has a local minimum at $u = x$ and thus so does $d^1_\eta(u) = 1 + d^0_\eta(u)$. Thus after shrinking $X$ further (so that $d^1_\eta(u) \congruent 1$) we may assume that $H^1(\eta)$ is flat, and $H^0(\eta)$ satisfies base change. Finally, this implies that $H^0(\eta) = (0)$.
\end{proof}

As noted in (RF\ref{axiom:classicality}), axiom (RF\ref{axiom:crystalline}) and Proposition \ref{prop:refinements-triangulations} imply that for each $x_0 \in X_{\cl}$ there exists a triangulation $P_{x_0,\bullet}$ of $D_{x_0}$ whose parameter $(\twid \delta_{1,x_0},\dotsc,\twid \delta_{d,x_0})$ is homothetic to $(\delta_{1,x_0},\dotsc,\delta_{d,x_0})$. Moreover, if $x_0 \in X_{\cl}^{\varphi_{\reg}}$ then Lemma \ref{lemma:generic-cohomology} implies that $P_{x_0,\bullet}$ is the unique such triangulation of $D_{x_0}$. Thus we may unambiguously refer to the parameter $(\twid \delta_{1,x_0},\dotsc,\twid \delta_{d,x_0})$ for $x_0 \in X_{\cl}^{\varphi_{\reg}}$.

\begin{proposition}\label{prop:pointwise-existence}
If $x_0 \in X_{\cl}^{\varphi_{\reg}}$ then there exists an open affinoid neighborhood $x_0 \in U \ci X$ such that $\restrict{D}{U}$ is a pointwise triangulated family with torsion center $x_0$ (but without actual torsion) and given data
\begin{itemize}
\item torsion weights $(\wt_{\tau}(\twid \delta_{1,x_0}),\dotsc,\wt_{\tau}(\twid \delta_{d,x_0}))_{\tau}$,
\item parameter $(\delta_1,\dotsc,\delta_d)$,
\item non-critical points $X_{\cl}^{\nc,\varphi_{\reg}}$, and
\item center $x_0$.
\end{itemize}
\end{proposition}
\begin{proof}
We've been given the data in the statement of the proposition and so our task is to verify the axioms (TF\ref{axiom:tf-weight}) -- (TF\ref{axiom:tf-accumulates}). The axiom (TF\ref{axiom:tf-weight}) is clear by (RF\ref{axiom:crystalline}) and the definition of the parameter $(\twid \delta_{1,x_0},\dotsc,\twid \delta_{d,x_0})$. Next, we may shrink $X$ so that if $i < j$ then $\delta_i^{-1}\delta_j$ is everywhere generic up to homothety by Lemma \ref{lemma:generic-cohomology}, giving (TF\ref{axiom:tf-generic}).

The axioms (TF\ref{axiom:tf-trix}) -- (TF\ref{axiom:tf-accumulates}) are easily verified by the remarks preceding the theorem, and (RF\ref{axiom:classicality}). Thus it remains to check (TF\ref{axiom:tf-triangulation}), i.e. that each point $x$ is triangulated by a triangulation whose parameter is homothetic to $(\delta_{1,x},\dotsc,\delta_{d,x})$. But if $u \in X_{\cl}^{\nc}$ then this is true at $u$ and, moreover, by definition of non-critical we have that $D_u$ has ordered parameter $(\delta_{1,u},\dotsc,\delta_{d,u})$ on the nose. As we've already verified, there is a unique triangulation of $D_u$ whose parameter is homothetic to $(\delta_{1,u},\dotsc,\delta_{d,u})$ up to homothety. Thus the $(\varphi,\Gamma_K)$-module $D$ over $X$ is densely pointwise strictly trianguline with ordered parameter $(\delta_1,\dotsc,\delta_d)$ in the sense of \cite[Definition 6.3.2]{KedlayaPottharstXiao-Finiteness}. The existence of the triangulation demanded by (TF\ref{axiom:tf-triangulation}) is deduced from  \cite[Theorem 6.3.13]{KedlayaPottharstXiao-Finiteness}.
\end{proof}

Continue to let $D$ be a refined family of $(\varphi,\Gamma_K)$-modules and let $x \in X_{\cl}^{\varphi_{\reg}}$. Using Proposition \ref{prop:pointwise-existence} we assume, by shrinking $X$, that for all $u \in X$, the $(\varphi,\Gamma_K)$-module $D_u$ has a unique triangulation whose parameter $(\twid \delta_{1,u},\dotsc,\twid \delta_{d,u})$ is homothetic to $(\delta_{1,u},\dotsc,\delta_{d,u})$.

\begin{definition}
Let $u \in X(\bar \Q_p)$. The canonical (with respect to the refined family $D$) triangulation is the unique triangulation of $D_u$ whose parameter  is homothetic to $(\delta_{1,u},\dotsc,\delta_{d,u})$.
\end{definition}

Note that $u \mapsto \twid \delta_{i,u}$ does not, in cases of interest, glue to define a continuous character $\twid \delta_i : K^\x \goto A^\x$. In fact that will essentially only happen at points $u$ where $\twid \delta_{i,u} = \delta_{i,u}$. The best one can hope for is that a certain subparabolization of the canonical triangulation does analytically vary over $X$.

\begin{definition}\label{defn:non-critical-two}
Let $x \in X(\bar \Q_p)$ and $P_{x,\bullet} = (P_{x,i})$ be the canonical triangulation. We say that $P_{x,i}$ is non-critical if $\delta_{1,x}\dotsb \delta_{i,x} = \twid \delta_{1,x}\dotsb \twid \delta_{i,x}$. If
\begin{equation*}
I^{\nc}_x := \set{i \st P_{x,i} \text{ is non-critical}} = \set{0 = i_0 < i_1 < i_2 < \dotsb < i_s = d}
\end{equation*}
is the set of non-critical indices then we define the maximal non-critical parabolization $P_{x,\bullet}^{\nc}$ of $D_x$ by 
\begin{equation*}
P_{x,j}^{\nc}: 0 \sci P_{x,i_0} \sci P_{x,i_1} \sci \dotsb \sci P_{x,i_s} = D_x.
\end{equation*}
\end{definition}
If $x \in X_{\cl}$ then a comparison of Hodge--Tate weights implies that the previous definition agrees with the one(s) given in Section \ref{subsect:critical-noncritical}.

Suppose that $x \in X_{\cl}$. If $x$ is non-critical then one knows that the $\tau$-Hodge--Tate weights of $P_{x,i}$ are $\set{\kappa_{1,\tau}(x),\dotsc,\kappa_{i,\tau}(x)}$ by Lemma \ref{lem:equivalent-noncritical} and the definition of non-critical. Thus, for general $x \in X_{\cl}$, Sen's theory of Hodge--Tate weights in families \cite{Sen-VariationHodgeStructure} implies that one can only hope for $P_{x,i}$ to extend to an affinoid neighborhood provided $\HT_\tau(P_{x,i}) = \set{\kappa_{1,\tau}(x),\dotsc,\kappa_{i,\tau}(x)}$ for each $\tau$. That is, if we hope to spread $P_{x,i}$ out over a neighborhood then we need to know {\em a priori} that $i \in I_{x}^{\nc}$. Our main theorem is that the converse is true.

\begin{theorem}\label{thm:main-theorem}
If $D$ is a refined family of $(\varphi,\Gamma_K)$-modules over $X$ and $x_0 \in X_{\cl}^{\varphi_{\reg}}$ then there exists an affinoid neighborhood $x_0 \in U \ci X$ and a parabolization $P^{\nc}$ of $\restrict{D}{U}$ such that for each $u \in U = \Sp(B)$, the parabolization $P^{\nc}_{\bullet}\tensor_{B} L(u)$ of $D_u$ is a subparabolization of $P^{\nc}_{u,\bullet}$, with equality if $u = x_0$.
\end{theorem}

See the introduction for a history of this result.

\begin{proof}[Proof of theorem]
First, assume that $X$ is sufficiently small so that the conclusion of Proposition \ref{prop:pointwise-existence} holds. In particular, there is a canonical triangulation $P_{x,\bullet}$ at each point $x \in X$.

The proof will happen in three steps. By Proposition \ref{prop:pointwise-existence}, $D$ is a pointwise triangulated family with center $x_0$ but without torsion. Fix the unique $1 \leq n \leq d$ such that $P_{x_0,1}^{\nc} = P_{x_0,n}$. We assume that $n < d$, or else the theorem is proven already. It suffices to construct $P^{\nc}_1$ over an affinoid subdomain $U$ as in the statement of the theorem. Indeed, granting its existence, $P^{\nc}_1$ is a $(\varphi,\Gamma_K)$-module of rank $n$ (since it has rank $n$ at $x_0$) and thus after replacing $X$ by $U$ we have $P^{\nc}_1 \tensor_A L(x) = P_{x,n}$ for all $x \in X$. But then the quotient $D/P^{\nc}_1$ is a refined family of $(\varphi,\Gamma_K)$-modules over $X$ whose global parameter is $(\delta_{n+1},\dotsc,\delta_d)$. And so inductively we can apply the construction of $P_1^{\nc}$ we are about to give, if necessary.

Now we focus on constructing $P^{\nc}_1$. For each $\tau$, fix the unique integer $n_\tau$ such that $n_\tau$ such that $\kappa_{1,\tau}(x_0) = \wt_\tau(\twid \delta_{n_\tau,\tau})$. Note $n_\tau \leq n$ because $P_{x_0,1}^{\nc} = P_{x_0,n}$ is non-critical. Let $Q_0 = D$ and $X_0 = X$. 

\begin{claim}[Step 1]
There exists a sequence of affiniod subdomains $X = X_0 \cg X_1 \cg \dotsb \cg X_n$ and nearly flat pointwise triangulated families $Q_i$ over $X_i$ with finite cohomology and torsion center $x_0$ such that there is an exact sequence
\begin{equation*}
0 \goto \prod_\tau t_\tau^{\kappa_{i+1,\tau}(x_0) - \kappa_{1,\tau}(x_0)} \cal R_{X_{i+1}}(\delta_{i+1}) \goto \restrict{Q_i}{X_{i+1}} \goto Q_{i+1} \goto 0,
\end{equation*}
of generalized $(\varphi,\Gamma_K)$-modules over $X_{i+1}$ and the invariants of $Q_i$ are given by
\begin{itemize}
\item torsion weights $\left((\wt_\tau(\twid \delta_{j,x_0}))_\tau\right)_{1 \leq j \leq d}$
\item parameter $(\delta_1,\dotsc,\delta_d)$ and
\item non-critical points $X_{i}^{\nc} = X_i \intersect X_{\cl}^{\nc}$.
\end{itemize}
Moreover, $Q_{i,x}$ has free length $d-i \geq d-n > 0$.
\end{claim}
\begin{proof}[Proof of Step 1]To prove the claim, one easily argues by induction on $i$ using Theorem \ref{thm:make-torsion}. 
\end{proof}

Since $n_\tau \leq n$ for each $\tau$ we see from the choice of torsion weights and the definition of $n_\tau$ that $Q_{x_0,n}$ has $\tau$-torsion length $\ell_\tau(Q_{x_0,n}) = n-1$. On the other hand, if $u \in X_{n}^{\nc}$ then $\ell_\tau(Q_{u,n}) = n-1$ as well. Thus the $\tau$-torsion lengths at $x_0$ agree with the $\tau$-torsion lengths on a set of accumulating at $x_0$. We will now kill the torsion. Let $Q_0' = Q_n$ and $X_0' = X_n$ as in Step 1.

\begin{claim}[Step 2] There exists a nested sequence of affinoid subdomains $X_n = X_1' \cg \dotsb \cg X_n'$ and nearly flat pointwise triangulated families $Q_i'$ over $X_i'$ with torsion center $x_0$ such that, for $2\leq i < n$, there is a short exact sequence
\begin{equation*}
0 \goto \restrict{Q_i'}{X_{i+1}'}\left[\prod_\tau t_\tau^{\kappa_{i+1,\tau}(x_0) - \kappa_{i,\tau}(x_0)}\right] \goto \restrict{Q_i'}{X_{i+1}'} \goto Q_{i+1}' \goto 0
\end{equation*}
of generalized $(\varphi,\Gamma_K)$-modules over $X_{i+1}'$ and:
\begin{itemize}
\item The invariants of $Q_{i}'$ are:
\begin{itemize}
\item torsion weights $\left( (\wt_\tau(\twid \delta_{j,x_0}))_\tau\right)_{1 \leq j \leq d \st\wt_\tau(\twid \delta_{j,x_0}) \geq \kappa_{i,\tau}(x_0)}$
\item parameter $(\delta_{i},\dotsc,\delta_d)$ and
\item non-critical points $(X_{i}')^{\nc} = X_{n}^{\nc} \intersect X_{i}'$.
\end{itemize}
\item $Q_i'$ has free length $d-n$, independent of $i$ and
\item For each $\tau$ and $u \in (X_{i}')_{\cl}^{\nc,\varphi_{\reg}}$, the $\tau$-torsion lengths $\ell_\tau(x_0)$ and $\ell_\tau(u)$ are equal (both) to $n-i$.
\end{itemize}
\end{claim}
\begin{proof}[Proof of Step 2]
The claimed properties for $Q_1' = Q_n$ follow from the conclusion of Step 1. If $1 < i \leq n$ then the existence of $Q_i'$ over $X_i'$, with the given invariants, is proved by Corollary \ref{cor:tau-balanced}. The $\tau$-torsion lengths at $u \in (X_i')^{\nc,\varphi_{\reg}}_{\cl}$ are easily seen to be $n-i$, so to finish this step we just need to compute the $\tau$-torsion length at $x_0$. To do that we look at the non-zero $\tau$-torsion exponents in the standard triangulation of $Q_{i}'$. The torsion exponents, are by definition, given by
\begin{equation}\label{eqn:step2-torsion-wts}
\set{\wt_\tau(\twid\delta_{j,x_0}) - \wt_\tau(\delta_{i,x_0}) \st 1 \leq j \leq n \text{ and } \wt_\tau(\twid\delta_{j,x_0}) \geq \wt_\tau(\delta_{i,x_0}) }.
\end{equation}
Since $P_{x_0,n}$ is a non-critical step in the canonical triangulation at $x_0$, the set of weights $\set{\wt_\tau(\twid\delta_{1,x_0}),\dotsc,\wt_\tau(\twid\delta_{n,x_0})}$ are the lowest $n$ weights and so we see immediately that there are $\ell_\tau(x_0) = n-i$ non-zero elements in the set \eqref{eqn:step2-torsion-wts}.
\end{proof}

\begin{claim}[Step 3]
Finally we set $C_n = Q_n'$ and $U = X_n' = \Sp(B)$ as in Step 2. We claim that we can shrink $U$ so that $P^{\nc}_1 := \ker(D \surject C_n)$ is a $(\varphi,\Gamma_K)$-module over $U$ and for all $u \in U$,  $P^{\nc}_1\tensor_B L(u) = P_{u,n}$.
\end{claim}
\begin{proof}[Proof of Step 3]
Consider the pointwise triangulated family $C_n=Q_n'$ over $U = X_n'$. By Step 2, $C_n$ is torsion free at $x = x_0$ or $x = u \in U^{\nc}$. Thus after shrinking $U$, applying Lemma \ref{lemma:ranks-by-t}, we can assume that each fiber $C_{n,u}$ is finite free over $\cal R_{L(u)}$ of rank $d-n$, independent of $u$. Thus $C_n$ is a $(\varphi,\Gamma_K)$-module over $U$ (spread $C_n$ out to an open annulus and use \cite[Corollary 2.1.7]{KedlayaPottharstXiao-Finiteness}). Defining $P_1^{\nc}$ as the kernel of the natural surjection $D \surject C_n$ (note that all the constructions in Steps 1 and 2 were quotients), we get a $(\varphi,\Gamma_K)$-module and thus our candidate $P_1^{\nc}$. 

It remains to compute $P_1^{\nc}\tensor_B L(u)$ as a $(\varphi,\Gamma_K)$-submodule of $D_u$ for each $u$. But we've assumed throughout that $X$ was sufficiently small so that for all $x \in X$ there was a unique triangulation with parameter $(\delta_{1,x},\dotsc,\delta_{d,x})$ up to homothety. So, in order to check $P_1^{\nc}\tensor_B L(u) = P_{u,n}$ it is enough to show that $C_{n,u}$ can be triangulated by a parameter homothetic to $(\delta_{n+1,u},\dotsc,\delta_{d,u})$ for all $u \in U$. But  that latter claim follows from applying the information from Step 2 to $C_n = Q_n'$ and using axiom (TF\ref{axiom:tf-generic}).
\end{proof}
This completes the proof.
\end{proof}

\section{Ramification of weights}\label{sec:ramification}
Let $D$ be a refined family of $(\varphi,\Gamma_K)$-modules over a reduced $L$-affinoid space $X=\Sp(A)$ with parameter $(\delta_1,\dotsc,\delta_d)$ and classical points $X_{\cl}$. For each $i=1,\dotsc,d$ and $\tau \in \Sigma_K$, we consider the analytic functions $\kappa_{i,\tau}(x) := \wt_\tau(\delta_{i,x}) \in \Gamma(X,\cal O)$. The goal of this section is to study the infinitesimal differences $\kappa_{i,\tau} - \kappa_{j,\tau}$ at classical points.

Suppose that $x_0 \in X_{\cl}$, write $P_{x_0,\bullet}$ for its triangulation defined by axiom (RF\ref{axiom:crystalline}) and $(\twid \delta_{1,x_0},\dotsc,\twid \delta_{d,x_0})$ for the corresponding parameter. For each $\tau$, the list of Hodge--Tate weights $\set{\wt_\tau(\twid \delta_{1,x_0}),\dotsc,\wt_\tau(\twid \delta_{d,x_0})}$ must be the same as the list of integers $\set{\kappa_{1,\tau}(x_0),\dotsc,\kappa_{d,\tau}(x_0)}$. In particular, for each $\tau$ there is a permutation $\pi_{x_0,\tau}$ on $d$ letters such that $\wt_\tau(\twid \delta_{\pi_{x_0,\tau}(i),x_0}) = \kappa_{i,\tau}(x_0)$. To connect this with the non-critical jumps, the $\tau$-non-critical indices of $P_{x_0,\bullet}$ are exactly the integers $i$ such that $\pi_{x_0,\tau}$ restricts to a permutation on the set $\set{1,\dotsc,i}$. In particular, $\pi_{x_0,\tau}$ induces a permutation on the set of weights appearing in each non-critical step $P_{x_0,j}^{\nc}$ and thus permutes the set of weights appearing in each quotient $P_{x_0,j}^{\nc}/P_{x_0,j-1}^{\nc}$ as well.

If $B$ is a ring, let $B[\varepsilon] = B[T]/(T^2)$ be the ring of dual numbers. If $x_0 \in X$ then $\cal O^{\rig}_{X,x_0}$ denotes its local ring (in the rigid topology). Since $\cal O^{\rig}_{X,x_0}$ is Henselian, it contains a section of its residue field. We write $T_{x_0}X = \Hom_{L(x_0)}(\cal O^{\rig}_{X,x_0},L(x_0)[\varepsilon])$ for the Zariski tangent space at $x_0$. If $f \in A$ is a function on $X$ and $v \in T_{x_0}X$ is a tangent vector we write $\nabla_{v}(f) \in L(x_0)$ for the directional derivative of $f$ with respect to $v$. Explicitly it is given by $v(f) = f(x_0) + \nabla_v(f)\varepsilon$.

\begin{theorem}\label{theorem:ramification}
If $x_0 \in X_{\cl}^{\varphi_{\reg}}$ and $v \in T_{x_0}X$ then $\nabla_{v}( \kappa_{\pi_{x_0,\tau}(i),\tau} - \kappa_{i,\tau}) = 0$ for each $1 \leq i \leq d$ and $\tau \in \Sigma_K$.
\end{theorem}
As mentioned in the introduction, the theorem was discovered independently by the author and Breuil. A proof, similar to the one we are about to give, is given in \cite[Lemme 9.6 and Th\'eor\`eme 9.7]{Breuil-LocallyAnalyticSocle2}. We've taken an extra effort to state a more precise result, dealing with all the weights and over a finite extension $K/\Q_p$.
\begin{proof}[Proof of Theorem \ref{theorem:ramification}]

Choose a uniformizer $\varpi_K \in K^\x$ and write, for $i = 1 ,\dotsc,d$, the character $\delta_i = \delta_i^{\wt}\delta_{i}^{\nr}$ as in Section \ref{subsec:rank-one}. Furthermore, write $\eta_i = \delta_1^{\wt}\dotsb \delta_i^{\wt}$ and $D_i := (\wedge^i D)(\eta_i^{-1})$. Then $D_i$ is a $(\varphi,\Gamma_K)$-module over $X$ and it has distinct lowest Hodge--Tate weight $0$ over $X_{\cl}$. Let $\Phi = \delta_1^{\nr}(\varpi_K)\dotsb \delta_i^{\nr}(\varpi_K) \in \Gamma(X, \cal O_X)^\x$. Then by \cite[Theorem 4.13]{Liu-Triangulations} the $(\cal O_X\tensor_{\Q_p} F)$-module $D_{\cris}(D_i)^{\varphi^{f_K} = \Phi}$ is locally free of rank one and satisfies base change.

Now let $v \in T_{x_0}X$. Write $\twid D_{i,v} := \nabla_v(D_i)$ for the deformation $D_i\tensor_{\cal O_{X,x_0},v} L(x_0)[\varepsilon]$ of $D_{i,x_0}$ in the tangent direction of $v$. By the result just mentioned, $D_{\cris}(\twid D_{i,v})^{\varphi^{f_K} = \Phi}$ is free of rank one over $(L(x_0)\tensor_{\Q_p} F)[\varepsilon]$. Let $E = \ker(\twid D_{i,v} \goto D_{i,x_0} \goto D_{i,x_0}/\eta_{i,x_0}^{-1}\det P_{x_0,i})$. By the lemma below we have that $E$ is a crystalline $(\varphi,\Gamma_K)$-module. In particular, it is Hodge--Tate.

But what are the Hodge--Tate weights of $E$? The $\tau$-Hodge--Tate weights of $D_{i,x_0}$ are given by
\begin{equation*}
\HT_\tau(D_{i,x_0}) = \set{ \sum_{j \in J} \kappa_{j,\tau}(x_0) - \sum_{j=1}^i \kappa_{j,\tau}(x_0) \st J \ci \set{1,\dotsc,d} \text{ and $\sizeof J = i$}}.
\end{equation*}
And, the $\tau$-Hodge--Tate--Sen weights of $\twid D_{i,v} = \wedge^i (\twid D_{1,v})$ (which are elements of $L(x)[\varepsilon]$) then are given by
\begin{equation*}
\set{ \left(\sum_{j \in J} \kappa_{j,\tau}(x_0)+\nabla_v(\kappa_{j,\tau})\varepsilon\right)  - \left(\sum_{j=1}^i \kappa_{j,\tau}(x_0)+\nabla_v(\kappa_{j,\tau})\varepsilon\right) \st J \ci \set{1,\dotsc,d} \text{ and $\sizeof J = i$}}.
\end{equation*}
We can reinterpret this by viewing $\twid D_{i,v}$ as a $(\varphi,\Gamma_K)$-module over $\cal R_{L(x_0)}$ of twice the rank. We see that the Sen operator $\Theta_{\sen}$ acting on $D_{\sen}(\twid D_{i,v})$ has a matrix built out of the blocks of the form
\begin{equation*}
M_J := \begin{pmatrix}
\sum_{j \in J} \kappa_{j,\tau}(x_0) - \sum_{j=1}^i \kappa_{j,\tau}(x_0) & \sum_{j \in J} \nabla_v(\kappa_{j,\tau}) - \sum_{j=1}^i \nabla_v(\kappa_{j,\tau})\\ 0 & \sum_{j \in J} \kappa_{j,\tau}(x_0) - \sum_{j=1}^i \kappa_{j,\tau}(x_0)
\end{pmatrix} \in \oper{Mat}_{2\x 2}(L(x_0)).
\end{equation*}
Now note that $\det P_{x_0,i}$ has Hodge--Tate weight $\kappa_{\pi_{x_0,\tau}(1)}(x_0) + \dotsb + \kappa_{\pi_{x_0,\tau}(i)}(x_0)$ and thus $\eta_{i,x_0}^{-1}\det P_{x_0,i}$ has Hodge--Tate weight $\sum_{j=1}^i \kappa_{\pi_{{x_0},\tau}(j),\tau}(x_0) - \kappa_{j,\tau}(x_0)$. By the short exact sequence
\begin{equation*}
0 \goto E \goto \twid D_{i,v} \goto D_{i,v}/\eta_{i,x_0}^{-1} \det P_{x_0,i} \goto 0
\end{equation*}
of $(\varphi,\Gamma_K)$-modules over $\cal R_{L(x_0)}$, the only $2\times 2$ block of $\restrict{\Theta_{\sen}}{D_{\sen}(\twid D_{i,v})}$ which appears in $\restrict{\Theta_{\sen}}{D_{\sen}(E)}$ is the block
\begin{equation*}
\begin{pmatrix}
\sum_{j=1}^i \kappa_{\pi_{{x_0},\tau}(j),\tau}(x_0) - \kappa_{j,\tau}(x_0) & \sum_{j=1}^i \nabla_v(\kappa_{\pi_{{x_0},\tau}(j),\tau} - \kappa_{j,\tau})\\ 0 & \sum_{j=1}^i \kappa_{\pi_{{x_0},\tau}(j),\tau}(x_0) - \kappa_{j,\tau}(x_0)
\end{pmatrix}.
\end{equation*}
corresponding to $J = \set{\pi_{x_0,\tau}(1),\dotsc,\pi_{x_0,\tau}(i)}$. Finally since $E$ is crystalline, it is Hodge--Tate and thus $\Theta_{\sen}$ acts semi-simply. In particular, we conclude that
\begin{equation*}
\sum_{j=1}^i \nabla_v(\kappa_{\pi_{x_0,\tau}(j),\tau} - \kappa_{j,\tau}) = 0.
\end{equation*}
Since this is for any $i$, we conclude the theorem by induction on $i$.
\end{proof}
It remains to give the computation left unresolved in the previous proof.
\begin{lemma}\label{lemma:ramification-lemma}
Suppose that $D$ is a crystalline $(\varphi,\Gamma_K)$-module over $\cal R_L$, $\phi \in L^\x$ and $P_1 \ci D$ is a rank one saturated submodule such that $D_{\cris}(D)^{\varphi^{f_K} = \phi} = D_{\cris}(P_1)^{\varphi^{f_K} = \phi}$ is free of rank one over $L\tensor_{\Q_p} F$. If $\twid D \in \Ext^1(D,D)$ is an extension such that $D_{\cris}(\twid D)^{\varphi^{f_K} = \twid \phi}$ is free of rank one over $(L\tensor_{\Q_p} F)[\varepsilon]$ for some $\twid \phi \congruent \phi \mod \varepsilon$ then the image of $\twid D$ under the natural map $\Ext^1(D,D) \goto \Ext^1(P_1,D)$ is a crystalline $(\varphi,\Gamma_K)$-module over $\cal R_L$.
\end{lemma}
Note that the image of $\twid D$ is given by the $(\varphi,\Gamma_K)$-module $E = \ker(\twid D \goto D/P_1)$. Thus the lemma actually fills the gap left in the previous proof.
\begin{proof}
 If $T$ is a $B$-linear operator on a $B$-module $M$ and $b \in B$ then we let $M^{(T=b)}$ denote the submodule of elements $m \in M$ such that $(T-b)^n m = 0$ for for some $n\geq 0$. The functor $M \mapsto M^{(T=b)}$ is exact.

Since $\varphi^{f_K}$ is linear, $\phi$ is a simple eigenvalue for $\varphi^{f_K}$ and $D_{\cris}(-)$ is left exact, we see that $D_{\cris}(E)^{(\varphi^{f_K} = \phi)} = D_{\cris}(\twid D)^{(\varphi^{f_K} = \phi)}$. In particular, since $D_{\cris}(\twid D)^{\varphi^{f_K} = \twid \phi} = D_{\cris}(\twid D)^{(\varphi^{f_K} = \phi)}$ we see that $\dim_L D_{\cris}(E)^{(\varphi^{f_K} = \phi)} \geq 2\dim_{\Q_p} F$. Now consider the exact sequence 
\begin{equation*}
0 \goto D_{\cris}(D)^{(\varphi^{f_K} = \phi)} \goto D_{\cris}(E)^{(\varphi^{f_K} = \phi)} \goto D_{\cris}(P_1)^{(\varphi^{f_K} = \phi)}.
\end{equation*}
By counting dimensions we see the final map is surjective as well. Since $P_1$ is rank one, $D_{\cris}(P_1)^{(\varphi^{f_K} = \phi)} = D_{\cris}(P_1)$. In particular, $D_{\cris}(E) \goto D_{\cris}(P_1) \goto 0$ is exact as well, meaning that $E$ is crystalline (again by dimension counts).
\end{proof}

\begin{remark}
One can interpret Theorem \ref{theorem:ramification}, and Lemma \ref{lemma:ramification-lemma}, as making a statement about a certain deformation ring of $(\varphi,\Gamma_K)$-modules. Indeed, let $\frak X_{D_{x_0}}$ denote the functor of formal deformations of $D_{x_0}$. Let $R_{x_0,\bullet} = (\phi_1,\dotsc,\phi_d)$ be the refinement corresponding to the triangulation $P_{x_0,\bullet}$ by Proposition \ref{prop:refinements-triangulations}. One may define a relatively representable subfunctor $\frak X_{D_{x_0}} \cg \frak X^{h}_{D_{x_0}, R_{x_0,\bullet}}$ consisting of deformations $\twid D$ whose successive exterior powers $\wedge^i \twid D$ contain a free rank one submodule on which $\varphi^{f_K}$ acts by $\phi_1\dotsc \phi_i$. Neither functor $\fr X_{?}$ is in general representable but both satisfy the natural Mayer--Vietoris condition on $L(x_0)[\varepsilon]$-points in order to have reasonable Zariski tangent spaces $\fr X_{?}(L(x_0)[\varepsilon])$ (see \cite{Kisin-OverconvergentModularForms, Tan-Thesis}). In that case $\fr X_{D_{x_0}}(L(x_0)[\varepsilon]) \simeq \Ext^1(D_{x_0},D_{x_0})$. What we just showed is that the obvious differences of Hodge--Tate weights are constant over the tangent space to $\fr X_{D_{x_0},R_{x_0,\bullet}}^{h}$.
\end{remark}

\begin{remark}
One might also ask if Theorem \ref{theorem:ramification} provides a tight bound for the rank of the weight map in a refined family. The answer is no. For example, let $K = \Q_p$ and consider any refined family $D$ of rank two with global parameter $(\delta_1,\delta_2)$. Then one can takes its symmetric square $\Sym^2 D$ equipped with the structure of a refined family of rank three, naturally having global parameter $(\delta_1^2, \delta_1\delta_2,\delta_2^2)$. If $x_0$ defines a critically triangulated classical point for $D$ then its critical-type is the permutation $(12)$, and $x_0$ is also critically triangulated in the family $\Sym^2 D$ with critical-type $(13)$. Theorem \ref{theorem:ramification} only implies that the difference of the first weight (i.e. $2\kappa_1$) and the {\em third} weight (i.e. $2\kappa_2$) ramifies. However, the middle weight is $\kappa_1 + \kappa_2$ and so the difference between the first two weights is $2\kappa_1 - (\kappa_1 + \kappa_2) = \kappa_1 - \kappa_2$. This also ramifies in the family by Theorem \ref{theorem:ramification} applied to $D$ itself, but is not detected by Theorem  \ref{theorem:ramification} applied to just $\Sym^2 D$.
\end{remark}

\begin{appendices}

\section{Nakayama's lemma}
This brief appendix is to create a reference for a relative form of Nakayama's lemma we used in the main text. Throughout we let $K$ be a non-Archimedean field of characteristic zero which is complete with respect to a non-trivial absolute value. If $X$ and $Y$ are $K$-rigid spaces we denote by $\pr_X : X\times_K Y \goto X$ the projection map, which we note is obtained by base-changing the structure morphism $Y \goto \Sp(K)$. We begin with two lemmas on products of $K$-affinoid spaces.

\begin{lemma}\label{lemma:prod-facts}
Suppose that $X$ and $Y$ are $K$-affinoid spaces.
\begin{enumerate}
\item If $U \ci X\times_K Y$ is an affinoid open subdomain then the image $\pr_X(U) \ci X$ is a finite union of affinoid subdomains of $X$.
\item If $U \ci X\times_K Y$ is an admissible open then the image $\pr_X(U)\ci X$ is a union of affinoid open subdomains of $X$.
\end{enumerate}
\end{lemma}
\begin{proof}
We begin with part (a). The structure morphism $Y \goto \Sp(K)$ is flat and thus $\pr_X$ is flat also. Let $U \ci X\times_K Y$ be an affinoid open subdomain. The inclusion $U \ci X\times_K Y$ is also flat and thus the composition $U \goto X\times_K Y \goto X$ is a flat map between $K$-affinoid spaces. Part (a) then follows from \cite[Corollary 5.11]{Bosch-RigidGeometry2} because $K$-affinoids are quasi-compact and quasi-separated rigid spaces. To prove (b) we write $U = \bigunion U_i$ where each $U_i$ is an affinoid open subdomain. By (a) the image $\pr_X(U_i)$ is covered by affinoid subdomains and thus so is $\pr_X(U) = \bigunion \pr_X(U_i)$.
\end{proof}
\begin{remark}
We are unable to determine whether $\pr_X$ should be open in the sense that $\pr_X(U)$ is admissible open for each admissible open $U \ci X\times_K Y$. This contrasts with the algebraic analog of Lemma \ref{lemma:prod-facts}, where part (b) really becomes that $f$ is ``open" because ``covered by opens'' is synonymous with ``open'' in the Zariski topology (see \cite[Tag 037G]{Stacks}).
\end{remark}

If $X=\Sp(A)$ is a $K$-affinoid space and $x \in X$ we write $\ideal m_x$ for the corresponding maximal ideal of $A$. We write $\kappa(x)$ for the residue field $A/\ideal m_x$. If $F \in A$ and $x \in X$ then we use the standard notation $F(x)$ to denote the image of $F$ in $\kappa(x)$. If $A$ and $B$ are two $K$-affinoid algebras, $F \in A\hat\tensor_K B$ and $y \in \Sp(B)$ then we use $F_y$ to denote the element $F_y \in A\tensor_K \kappa(y)$ which is the image of $F$ under the canonical map $A\hat\tensor_K B \goto (A\hat\tensor_K B)/\ideal m_y(A\hat\tensor_K B) = A\tensor_K \kappa(y)$. 

\begin{lemma}\label{lemma:reduced}
Suppose that $A$ and $B$ are reduced affinoid $K$-algebras. Then:
\begin{enumerate}
\item The completed tensor product $A\hat\tensor_K B$ is reduced.
\item If $F \in A\hat\tensor_K B$ then $F = 0$ if and only if $F_y = 0$ for all $y \in \Sp(B)$.
\end{enumerate}
\end{lemma}
\begin{proof}
A noetherian ring $R$ is reduced if and only if satisfies the two properties $(R_0)$ and $(S_1)$ \cite[Tag 031R]{Stacks}, and being geometrically reduced is equivalent to being reduced for affinoid algebras over perfect fields \cite[Lemma 3.3.1]{Conrad-IrredComponents}. In particular, since $K$ has characteristic zero we deduce part (a) from \cite[Th\'eor\`eme 8.1]{Ducros-BerkovichSpaces}.\footnote{The reference \cite{Ducros-BerkovichSpaces} is written for Berkovich spaces, but the reducedness of $A\hat\tensor_K B$ can be checked with either the rigid analytic space or its associated Berkovich space. Indeed, the completions of the analytic local rings are the completions of the algebraic local rings in either case, and reducedness can be checked after completion by the excellence of $K$-affinoid algebras (see \cite[Section 2.2]{Bekovich-EtaleCohomology} for example).}

To prove part (b) we let $y \in \Sp(B)$ and note that if $w \in X\times_K Y$ lies above $y$ then we have a natural map $A\tensor_K \kappa(y) \goto (A\hat\tensor_K B)/\ideal m_{w}$. Ranging over all $y$ we get a commutative diagram
\begin{equation}\label{eqn:reduced-diagram}
\xymatrix{
A\hat\tensor_K B \ar[r] \ar[d]& \prod_{y \in Y} A \tensor_K \kappa(y) \ar[dl]\\
\prod_{w \in X\times_K Y} (A\hat\tensor_K B)/\ideal m_w
}
\end{equation}
By part (a) the vertical arrow of \eqref{eqn:reduced-diagram} is injective. Thus we deduce the horizontal arrow of \eqref{eqn:reduced-diagram} is injective as well, proving part (b).
\end{proof}

\begin{lemma}\label{lemma:Zdense}
If $X=\Sp(A)$ and $Y=\Sp(B)$ are reduced affinoid $K$-spaces and $Z \ci X$ is a Zariski dense subset of points then $\pr_X^{-1}(Z) \ci X\times_{K} Y$ is also Zariski dense.
\end{lemma}
\begin{proof}
Suppose that $Z \ci X$ is Zariski dense. Let $F \in A\hat{\tensor}_K B$. We need to show that if $F(z') = 0$ for all $z' \in \pr_X^{-1}(Z)$ then $F = 0$. It suffices by Lemma \ref{lemma:reduced}(b) to show that $F_y = 0$ for each $y \in Y$.

Let $y \in Y$. The field $\kappa(y)$ is a finite extension of $K$. In particular, if we fix a basis of $\kappa(y)$ over $K$ then for any $K$-vector space $C$ we get an identification of $C\tensor_K \kappa(y)$ with $C^{\dsum n}$ (where $n = \dim_K \kappa(y)$) which is functorial in $C$ (but depending on the choice of basis). Applying this to $C = A$ and $C = \kappa(z)$ with $z \in Z$ we get a commuting diagram whose vertical arrows are isomorphisms
\begin{equation*}
\xymatrix{
A\tensor_K \kappa(y) \ar[d]_-{\iso} \ar[r] &  \prod_{z \in Z} \kappa(z)\tensor_K \kappa(y) \ar[d]^-{\iso}\\
A^{\dsum n} \ar[r] & \prod_{z \in Z} \kappa(z)^{\dsum n}}
\end{equation*}
Since $Z$ is Zariski dense in $X$, the bottom arrow is injective. Thus to show $F_y = 0$ it suffices to check that its image in $\kappa(z)\tensor_K \kappa(y)$ is zero for each $z \in Z$. But that follows easily from the assumption that $F(z') = 0$ for all $z' \in \pr_X^{-1}(Z)$. Since $y$ was arbitrary, we are done.
\end{proof}

If $B$ is a ring and $f \in B$ is non-nilpotent then we write $B[1/f]$ for the localization of $B$ at $f$. If $M$ is a $B$-module we write $M[1/f] = B[1/f]\tensor_B M$ for the localization of $M$. If $A$ is a $K$-affinoid algebra, $M$ is a module over $A$ and $x \in X$ corresponds to the maximal ideal $\ideal m_x$ then we write $M_x$ for $M\tensor_A A/\ideal m_x$. Note that these two operations commute in the sense that $M_x[1/f] \simeq M[1/f]_x$. 

Now suppose that $A$ and $B$ are affinoid $K$-algebras and $R = A\hat\tensor_K B$. If $M$ is a module over $R$ then $M_x[1/f]$ is a module over $\kappa(x)\tensor_K B[1/f]$ for each $x \in X$ and $f \in B$. Indeed, one just has to check that 
\begin{align}
R_x[1/f] &\simeq \left((B\hat\tensor_K A)\tensor_A A/\ideal m_x\right)\tensor_B B[1/f] \label{eqn:tensor-manip}\\
&\simeq \left(B\tensor_K \kappa(x)\right)\tensor_B B[1/f]\nonumber\\
&\simeq \kappa(x)\tensor_K B[1/f]\nonumber
\end{align}
We now arrive at the subject of this appendix.
\begin{proposition}[Nakayama's lemma]\label{prop:nakayama-lemma}
Suppose that $X = \Sp(A)$ and $Y = \Sp(B)$ are reduced affinoid $K$-spaces and that $M$ is a finite module over $A\hat\tensor_K B$. Let $f \in B$ be non-zero and assume that for each $x \in X$ the $\kappa(x)\tensor_K B[1/f]$-module $M_x[1/f]$ is (finite) free. Then:
\begin{enumerate}
\item For each $m \geq 0$ the subset
\begin{equation*}
X_m = \set{x \in X \st \rank_{\kappa(x)\tensor_K B[1/f]} M_x[1/f] \leq m} \ci X
\end{equation*}
is a union of affinoid subdomains of $X$.
\item If $Z \ci X$ is Zariski dense then $X_m \neq \es \implies Z \intersect X_m \neq \es$. In particular, the minimum rank of $M_x[1/f]$ is achieved on $Z$.
\end{enumerate}
\end{proposition}
\begin{proof}
Let $R = A\hat\tensor_K B$ and consider the (non-empty, because $f$ is not nilpotent) affine scheme $\Spec(R[1/f]) \ci \Spec(R)$. If $\ideal q \in \Spec(R)$ write $\kappa(\ideal q)$ for its residue field (generalizing our previous notation). Since $M$ is finite over $R$, $M[1/f]$ is finite over $R[1/f]$ and so by the usual Nakayama's lemma (see \cite[Theorem 4.10]{Matsumura-CommutativeRingTheory}) the set
\begin{equation*}
V_m = \set{\ideal q \in \Spec(R[1/f]) \st \dim_{\kappa(\ideal q)} M[1/f] \tensor_{R[1/f]} \kappa(\ideal q) \leq m}
\end{equation*}
is Zariski open in $\Spec(R[1/f])$ for each $m\geq 0$. In particular, $U_m := V_m \intersect \Sp(R)$ is Zariski open in $\Sp(R) = X\times_K Y$. Since Zariski opens are admissible opens, $U_m$ is admissible open in $\Sp(R)$. By Lemma \ref{lemma:prod-facts}(b), $\pr_X(U_m)$ is a union of affinoid subdomains and thus it suffices to show that $\pr_X(U_m) = X_m$.

Write $\Sp(R)_f$ for $\Sp(R) \intersect \Spec(R[1/f])$. Since $f \in B$, the projection map $\Sp(R)_f \goto X$ is still surjective. Thus if $x \in X$ we may choose $u \in \Sp(R)_f$ lying above $x$. Then \eqref{eqn:tensor-manip} shows that $R_x[1/f]\tensor_{\kappa(x)\tensor_K B[1/f]} \kappa(u) = \kappa(u)$. We're assuming $M_x[1/f] \simeq (R_x[1/f])^{\dsum n_x}$ for an integer $n_x$ depending on $x$. Reducing the residue field at $u$, we see that
\begin{equation*}
M_x[1/f]\tensor_{\kappa(x)\tensor_K B[1/f]} \kappa(u) \simeq \kappa(u)^{\dsum n_x}.
\end{equation*}
On the other hand, since
\begin{equation*}
M_x[1/f]\tensor_{\kappa(x)\tensor_K B[1/f]} \kappa(u) \simeq M[1/f]\tensor_{R[1/f]} \kappa(u)
\end{equation*}
this shows that
\begin{equation*}
\dim_{\kappa(u)} M[1/f] \tensor_{R[1/f]} \kappa(u) = \rank_{\kappa(x)\tensor_K B[1/f]} M_x[1/f]
\end{equation*}
depends only on $x$. Now it is clear that $\pr(U_m) \ci X_m$. The reverse inclusion follows from the surjectivity of $\Sp(R)_f \goto X$. This shows $\pr_X(U_m) = X_m$ and we have finished the proof of (a).

The proof of (b) is nearly complete as well. The only point is that if $Z$ is Zariski dense then Lemma \ref{lemma:Zdense} guarantees that the pre-image $\pr_X^{-1}(Z) \ci \Sp(R)$ is Zariski dense as well. In particular, since $U_m$ is {\em Zariski} open in $\Sp(R)$, $U_m$ is non-empty if and only if $\pr_X^{-1}(Z) \intersect U_m$ is non-empty. Part (b) follows immediately now.
\end{proof}

\end{appendices}

\bibliography{parab_bib}
\bibliographystyle{abbrv}

\end{document}